\documentclass[twoside]{aiml22}

\usepackage{aiml22macro}

\newenvironment{notation}{\vspace{-\lastskip}\par\addvspace{\thmvspace}\begin{conv}}{\end{conv}\par\addvspace{\thmvspace}}
\newtheorem{conv}[thm]{Notation}

\usepackage{graphicx}
\usepackage{amsmath}
\usepackage{amssymb}
\usepackage{mathtools}
\usepackage[colorlinks=true]{hyperref}
\usepackage{tikz}
\usetikzlibrary{positioning,calc,arrows}
\usepackage{tikz-cd}
\usepackage{xcolor}
\usepackage{hyperref}
\usepackage{enumerate}
\usepackage{longtable}
\newcommand{\FOL}{{\mathsf{FOL}}}

\newcommand{\HPL}{{\mathsf{HPL}}}
\newcommand{\HFOL}{{\mathsf{HFOL}}}
\newcommand{\HFOLS}{{\mathsf{HFOLS}}}
\newcommand{\RFOHL}{{\mathsf{RFOHL}}}
\newcommand{\Sig}{\mathtt{Sig}}
\newcommand{\Mod}{\mathtt{Mod}}
\newcommand{\Sen}{\mathtt{Sen}}
\newcommand{\Set}{\mathbb{S}et}
\newcommand{\Cat}{\mathbb{C}at}

\newcommand{\nom}{\mathtt{n}}
\newcommand{\rigid}{\mathtt{r}}
\newcommand{\flex}{\mathtt{f}}
\newcommand{\rt}{\mathtt{rt}}
\newcommand{\card}{\mathtt{card}}
\newcommand{\ext}{\mathtt{e}}
\newcommand{\ari}{\mathtt{ar}}
\newcommand{\inj}{\mathtt{inj}}

\newcommand{\red}{\!\upharpoonright\!}
\newcommand{\at}[1]{@_{#1}\,}

\newcommand{\pos}[1]{\langle #1 \rangle}
\newcommand{\store}[1]{{\downarrow}#1\,{\cdot}\,}
\newcommand{\Forall}[1]{\forall #1\,{\cdot}\,}
\newcommand{\Exists}[1]{\exists #1\,{\cdot}\,}

\def\lastname{G\u{a}in\u{a}, Badia and Kowalski}
\begin{document}
\begin{frontmatter}
  \title{Robinson consistency in many-sorted hybrid first-order logics}
  \author{Daniel G\u{a}in\u{a}}\footnote{The work presented in this paper has been partially supported by Japan Society for the Promotion of Science, grant number 20K03718.}
  \address{Kyushu University}
  
  \author{Guillermo Badia}
  \address{University of Queensland}
  
  \author{Tomasz Kowalski}
  \address{Jagiellonian University}

\begin{abstract}
In this paper we prove a Robinson consistency theorem for a class of many-sorted hybrid logics as a consequence of an  Omitting Types Theorem.
An important corollary of this result is an interpolation theorem.
\end{abstract}

  \begin{keyword}
Institution, Hybrid logic, Robinson consistency, Interpolation.
  \end{keyword}
 \end{frontmatter}
\section{Introduction}

Robinson's Joint Consistency Theorem \cite{rob} gives a sufficient condition in
the context of first-order logic for two theories to have a common model.  This
result was originally proved by A. Robinson  with the aim of providing a new
purely model-theoretic proof of the Beth definability property. Robinson's
theorem was a historical forerunner of Craig's celebrated interpolation
theorem, to which it is famously classically equivalent. 
In the context of
classical first-order logic, it is known since Lindstr\"om's work~\cite{lind78}
that in the presence of compactness, the Robinson Consistency Property is a
consequence of the Omitting Types Theorem. Following in Lindstr\"om's footsteps,
we use an Omitting Types Theorem for many-sorted hybrid-dynamic first-order logics established
in~\cite{gai-hott} to obtain a Robinson Consistency Theorem.
However, our results rely on compactness, so they apply to any star-free fragment of this logic.

In \cite{ArecesBM01}, Areces et al. solved the interpolation problem positively for hybrid
propositional logic, and in \cite{ArecesBM03}, they establish a similar result for
hybrid predicate logic with constant domains (also called here rigid domains).
Our results, although similar, do not follow
from theirs. For one thing, the framework of \cite{ArecesBM03} is limited to
constant domain quantification, whereas we allow variable domains. Moreover,
as usual in the area of algebraic specification, we work in a many-sorted
setting and, importantly, we consider arbitrary pushouts of signatures (see,
e.g.,~\cite{tar-bit}); not only inclusions. This seemingly small change
splits one-sorted interpolation and many-sorted interpolation apart. One can have
the former but not the latter, as Example \ref{ex:counter-2} shows. The same
holds for Robinson consistency. 

Our approach is based on institution theory, 
an abstract framework introduced in~\cite{gog-ins} for reasoning about properties of logical systems from a meta-perspective.
Institutional setting achieves generality appropriate for the development of
abstract model theory, yet it is geared towards applications, particularly
applications to specification and verification of systems. 
However, for the sake of simplicity,  we work in a concrete example of hybrid logic, 
applying the modularization principles advanced by institution theory, which have at the core the notion of signature morphism and the satisfaction condition (the truth is invariant w.r.t. change of notation).
This brings about certain peculiarities, such as regarding variables as special constants,
and omnipresence of signature morphisms, which are simply maps between signatures. 
For more on institution theory, we refer the reader to~\cite{dia-book}.

The article is structured in the usual way. Section~\ref{sec:HFOL} introduces
Hybrid First-Order Logic (HFOL) with rigid symbols in an institutional setting, giving
the expected definitions of Kripke structures, reducts, and (local and global)
satisfaction relation.  Sections~\ref{sec:basics} and~\ref{sec:relat} give some
preliminary results, most importantly the Lifting Lemma~\ref{lemma:lifting} and
Proposition~\ref{prop:amalg}. The bulk of the work is in Section~\ref{sec:rc}
where, unsurprisingly, we prove Robinson Joint Consistency Theorem for
HFOL. Our proof is modelled after Lindstr\"om~\cite{lind78}, and uses
our earlier result from~\cite{gai-hott} establishing Omitting Types Theorem
for HFOL.

\section{Hybrid First-Order Logic with rigid symbols ($\HFOL$)}\label{sec:HFOL}
In this section, we present hybrid first-order logic with rigid symbols.

\paragraph{Signatures}
The signatures are of the form $\Delta=(\Sigma^\nom,\Sigma^\rigid,\Sigma)$:
\begin{itemize}
\item $\Sigma=(S,F,P)$ is a many-sorted first-order signature such that
(a)~$S$ is a set of sorts, 
(b)~$F$ is a set of function symbols of the form $\sigma:\ari\to s$, where $\ari\in S^*$ is called the \emph{arity} of $\sigma$ and $s\in S$ is called the \emph{sort} of $\sigma$, and
(c)~$P$ is a set of relation symbols of the form $\pi:\ari$, where $\ari\in S^*$ is called the arity of $\pi$.~\footnote{$S^*$ denotes the set of all strings with elements from $S$.} 
\item $\Sigma^\rigid=(S^\rigid,F^\rigid,P^\rigid)$ is a many-sorted first-order
signature of \emph{rigid} symbols such that $\Sigma^\rigid\subseteq \Sigma$.
\item $\Sigma^\nom=(S^\nom,F^\nom,P^\nom)$ is a single-sorted first-order signature such that
$S^\nom=\{\nom\}$, 
$F^\nom$ is a set of constants called \emph{nominals}, and 
$P^\nom$ is a set of unary or binary relation symbols called \emph{modalities}.
\end{itemize}
We usually write  $\Delta=(\Sigma^\nom,\Sigma^\rigid\subseteq \Sigma)$ rather than $\Delta=(\Sigma^\nom,\Sigma^\rigid,\Sigma)$.
Throughout this paper, we let $\Delta$ and $\Delta^i$ 
range over signatures of the form $(\Sigma^\nom,\Sigma^\rigid\subseteq\Sigma)$ and $(\Sigma_i^\nom,\Sigma_i^\rigid\subseteq\Sigma_i)$, respectively.
A \emph{signature morphism} $\chi \colon \Delta \to \Delta^1$ consists of a pair of first-order signature morphisms
$\chi^{\nom} \colon \Sigma^{\nom} \to \Sigma_1^{\nom}$ and $\chi \colon \Sigma \to \Sigma_1$ such that $\chi(\Sigma^{\rigid}) \subseteq \Sigma_1^{\rigid}$.~\footnote{
A first-order signature morphism $\chi:(S,F,P)\to(S_1,F_1,P_1)$, is a triple $(\chi:S\to S_1, \chi:F\to F_1,\chi:P\to P_1)$ which maps each function symbol $\sigma:s_1\dots s_n\to s\in F$ to $\chi(\sigma):\chi(s_1)\dots \chi(s_n)\to\chi(s)\in F_1$ and each relation symbol $\pi:\ari\in P$ to $\chi(\pi):\chi(\ari)\in P_1$.}

\begin{fact}
$\HFOL$ signature morphisms form a category $\Sig^\HFOL$ under the component-wise composition as first-order signature morphisms.
\end{fact}

\paragraph{Kripke structures}
  For every signature $\Delta$, the class of Kripke structures over $\Delta$
  consists of pairs $(W,M)$, where
  
\begin{itemize}
\item $W$ is a first-order structure over $\Sigma^\nom$, called a \emph{frame},
with the universe $|W|$ consisting of a non-empty set of possible worlds, and
  
\item $M\colon|W|\to |\Mod^\FOL(\Sigma)|$ is a mapping from the universe of $W$ to the class of first-order $\Sigma$-structures such that the rigid symbols are interpreted in the same way across worlds:
${M_{w_1}\red_{\Sigma^\rigid}}={M_{w_2}\red_{\Sigma^\rigid}}$ for all $w_1,w_2\in |W|$, where
 $M_{w_i}$ denotes the first-order $\Sigma$-structure corresponding to $w_i$, and 
 $M_{w_i}\red_{\Sigma^\rigid}$ is the reduct of $M_{w_i}$ to the signature $\Sigma^\rigid$.
\end{itemize}
A \emph{homomorphism} $h \colon (W, M) \to (V,N)$ over a signature $\Delta$ is also a pair 
\begin{center}
$(W\stackrel{h}\to V, \{M_{w}\stackrel{h_{w}}\to N_{h(w)}\}_{w \in |W|})$
\end{center}
consisting of first-order homomorphisms such that the mappings corresponding to rigid sorts are shared across the worlds, that is, $h_{w_{1}, s} = h_{w_{2}, s}$ for all possible worlds $w_{1}, w_{2} \in |W|$ and all rigid sorts $s \in S^\rigid$.

\begin{fact}
For any signature $\Delta$, the $\Delta$-homomorphisms form a category $\Mod^\HFOL(\Delta)$ under the component-wise composition.
\end{fact}

\paragraph{Reducts}
Every signature morphism $\chi \colon \Delta \to \Delta^1$ induces appropriate \emph{reductions of models}:
every $\Delta^1$-model $(V, N)$ is reduced to a $\Delta$-model $(V,N) \red_{\chi}$ that interprets every symbol $x$ in $\Delta$ as $(V, N)_{\chi(x)}$.
When $\chi$ is an inclusion, we usually denote $(V, N) \red_\chi$ by $(V, N) \red_\Delta$ -- in this case, the model reduct simply forgets the interpretation of those symbols in $\Delta^1$ that do not belong to $\Delta$.
\begin{fact}
For each signature morphism
 $\chi\colon\Delta\to\Delta'$ and each Kripke structure $(W,M)$ over
 $\Delta'$, the map  $\Mod^\HFOL$ from $\Sig^\HFOL$ to $\Cat^{op}$, defined by
 $\Mod^\HFOL(\chi)(W,M) = (W,M)\red_\chi$, is a functor.
\end{fact}

\paragraph{Hybrid terms}
For any signature $\Delta$, we make the following notational conventions:
(a)~$S^\ext\coloneqq S^\rigid\cup\{\nom\}$ the extended set of rigid sorts, where $\nom$ is the sort of nominals,
(b)~$S^\flex \coloneqq S \setminus S^{\rigid}$ the subset of flexible sorts,  
(c)~$F^\flex\coloneqq F\setminus F^\rigid$ the subset of flexible function symbols, 
(d)~$P^\flex\coloneqq P\setminus  P^\rigid$ the subset of flexible relation symbols.
The \emph{rigidification} of $\Sigma$ with respect to $ F^\nom$ is the signature $@\Sigma=(@S,@F,@P)$, where 
(a)~$@S\coloneqq \{\at{k} s \mid k\in F^\nom \mbox{ and } s\in S\}$,
(b)~$@F\coloneqq\{\at{k}\sigma\colon \at{k}\ari \to \at{k} s \mid k\in F^\nom \mbox{ and } (\sigma\colon \ari\to s) \in F \}$, and
(c)~$@P\coloneqq \{\at{k} \pi\colon \at{k} \ari \mid k\in F^\nom \mbox{ and }(\pi\colon\ari)\in P\}$.~\footnote{$\at{k} (s_1\ldots s_n) \coloneqq \at{k} s_1\ldots\at{k} s_n$ for all arities $s_1\ldots s_n$.}
Since rigid symbols have the same interpretation across worlds, we let $\at{k} x = x$ for all nominals $k\in F^\nom$ and all rigid symbols $x$ in $\Sigma^\rigid$.
 The set of \emph{rigid $\Delta$-terms} is $T_{@\Sigma}$, while the set of \emph{open $\Delta$-terms} is $T_\Sigma$. 
 The set of \emph{hybrid $\Delta$-terms} is $T_{\overline\Sigma}$, where 
 $\overline\Sigma=(\overline{S},\overline{F},\overline{P})$, 
 $\overline{S}=S\cup @S^\flex$,
 $\overline{F}=F\cup @F^\flex$, and
 $\overline{P}=P\cup @P^\flex$.

The interpretation of the hybrid terms in Kripke structures is uniquely defined as follows:
for any $\Delta$-model $(W,M)$, and any possible world $w\in|W|$, 
\begin{enumerate}[1)]
\item $M_{w,\sigma(t)} = M_{w,\sigma}(M_{w,t})$, where $(\sigma\colon\ari\to s)\in F$, and $t\in T_{\overline\Sigma,\ari}$,
\footnote{$M_{w,(t_1,\ldots,t_2)}\coloneqq M_{w,t_1},\ldots,M_{w,{t_n}}$ for all tuples of hybrid terms $(t_1,\ldots,t_n)$.}
\item $M_{w,(\at{k} \sigma)(t)} = M_{w',\sigma} (M_{w,t})$, where $(\at{k} \sigma\colon\at{k} \ari\to\at{k} s)\in @F^\flex$, $t\in T_{\overline\Sigma,\at{k}\ari}$ and $w'=W_k$.
\end{enumerate}

\paragraph{Sentences}
The simplest sentences defined over a signature $\Delta$,
usually referred to as atomic, are given by
\begin{center}
$\rho \Coloneqq k \mid \varrho \mid t_{1} = t_{2} \mid \varpi(t)$
\end{center}
where 
(a)~$k \in  F^\nom$ is a nominal,
(b)~$(\varrho:\nom)\in P^\nom$ is a unary modality,
(c)~$t_i \in T_{\overline\Sigma,s}$ are hybrid terms, $s\in \overline{S}$ is a hybrid sort,
(d)~$\varpi:\ari\in\overline{P}$ and $t\in T_{\overline\Sigma,\ari}$.
We call 
\emph{hybrid equations} sentences of the form $t_1=t_2$, and  
\emph{hybrid relations} sentences of the form $\varpi(t)$.
The set $\Sen^\HFOL(\Delta)$ of \emph{full sentences} over $\Delta$ are given by the following grammar:
\begin{center}
$\varphi \Coloneqq
\rho \mid
\at{k} \varphi \mid
\lnot \varphi \mid
\textstyle\vee \Phi \mid
\store{z} \varphi' \mid
\Exists{X} \varphi'' \mid
\pos{\lambda} \varphi $
\end{center}
where 
(a)~$\rho$ is an atomic sentence,
(b)~$k \in  F^\nom $ is a nominal,
(c)~$\Phi$ is a finite set of sentences over $\Delta$,
(d)~$z$ is a nominal variable for $\Delta$ and $\varphi'$ is a sentence over the signature $\Delta(z)$ obtained from $\Delta$ by adding $z$ as a new constant to $ F^\nom$,
(e)~$X$ is a set of variables for $\Delta$ of sorts from the extended set of rigid sorts $S^\ext$ and $\varphi''$ is a a sentence over the signature $\Delta(X)$ obtained from $\Delta$ by adding the variables in $X$ as new constants to $ F^\nom$ and $F^{\rigid}$, and
(f)~$(\lambda:\nom~\nom)\in P^\nom$ is a binary modality.
Other than the first kind of sentences (\emph{atoms}), we refer to the sentence-building operators, as  
\emph{retrieve}, 
\emph{negation}, 
\emph{disjunction}, 
\emph{store}, 
\emph{existential quantification} and 
\emph{possibility}, 
respectively.
Other Boolean connectives and the universal quantification can be defined as abbreviations of the above sentence building operators.

Each signature morphism $\chi\colon\Delta\to\Delta^1$ induces \emph{sentence translations}:
any $\Delta$-sentence $\varphi$ is translated to a $\Delta^1$-sentence $\chi(\varphi)$ by replacing, in an inductive manner, the symbols in $\Delta$ with symbols from $\Delta^1$ according to $\chi$.

\begin{fact}
$\Sen^\HFOL$ is a functor $\Sig^\HFOL \to \Set$ which maps each signature $\Delta$ to the set of sentences over $\Delta$.
\end{fact}
\paragraph{Local satisfaction relation}
Given a $\Delta$-model $(W, M)$ and a world $w \in |W|$, we define the \emph{satisfaction of $\Delta$-sentences at $w$} by structural induction as follows:

\noindent\emph{For atomic sentences}:
\begin{itemize}
\item $(W, M) \models^{w} k $ iff $W_k = w$;
\item $(W,M)\models^w \varrho$ iff $w\in W_\varrho$;
\item $(W, M) \models^{w} t_{1} = t_{2}$ iff $M_{w, t_1} = M_{w,t_2}$;
\item $(W, M) \models^{w} \varpi(t)$ iff $M_{w,t} \in M_{w, \varpi}$.
\end{itemize}

\noindent\emph{For full sentences}:
\begin{itemize}
\item $(W, M) \models^{w} \at{k} \varphi$ iff $(W, M) \models^{w'} \varphi$, where $w' = W_{k}$;
\item $(W, M) \models^{w} \neg \varphi$ iff $(W, M) \not\models^{w} \varphi$;
\item $(W, M) \models^{w} \vee \Phi$ iff $(W, M) \models^{w} \varphi$ for some $\varphi \in \Phi$; 
\item $(W, M) \models^{w} \store{z}{\varphi}$ iff $(W^{z \leftarrow w}, M) \models^{w} \varphi$,
where $(W^{z \leftarrow w}, M)$ is the unique $\Delta(z)$-expansion of $(W, M)$ that interprets the variable $z$ as $w$;
\footnote{An expansion of $(W, M)$ to $\Delta(X)$ is a Kripke structure $(W', M')$ over $\Delta(X)$ that interprets all symbols in $\Delta$ in the same way as $(W, M)$.}    
\item $(W, M) \models^w \Exists{X}{\varphi}$ iff $(W', M') \models^w \varphi$ for some expansion $(W', M')$ of $(W, M)$ to the signature $\Delta(X)$;
\footnotemark[\thefootnote]
\item $(W, M) \models^{w} \pos{\lambda} \varphi$ iff $(W, M) \models^{w'} \varphi$ for some $w' \in |W|$ s.t. $(w, w') \in W_{\lambda}$.
\end{itemize}
The following \emph{satisfaction condition} can be proved by induction on the structure of $\Delta$-sentences.
The proof is essentially identical to those developed for several other variants of hybrid logic presented in the literature (see, e.g.~\cite{dia-qvh}).
\begin{proposition}[Local satisfaction condition] \label{prop:sat-cond}
Let $\chi \colon \Delta \to \Delta^1$ be a signature morphism.
Then $(W, M) \models^{w} \chi(\varphi)$ iff $(W, M) \red_{\chi} \models^{w} \varphi$,
for all Kripke structures over $\Delta^1$,
all sentences $\varphi$ over $\Delta$.~\footnote{By the definition of reducts, $(W', M')$ and $(W', M') \red_{\chi}$ have the same possible worlds, which means that the statement of Proposition~\ref{prop:sat-cond} is well-defined.}
\end{proposition}

\paragraph{Global satisfaction relation}
The global satisfaction relation is defined by 
\begin{center}
$(W,M)\models \varphi$ iff for each possible world $w\in|W|$ we have $(W,M)\models^w\varphi$
\end{center}
for all signatures $\Delta$, all Kripke $\Delta$-structures $(W,M)$ and all $\Delta$-sentences $\varphi$.
The global consequence relation between sentences is defined by
\begin{center} 
$\varphi\models \psi$ iff $(W,M)\models \varphi$ implies $(W,M)\models\psi$ for all Kripke structures $(W,M)$, 
\end{center}
for all sentences $\varphi$ and $\psi$ over the same signature. 
The global consequence relation can be extended to sets of sentences in the usual way.

We adopt the terminology used in the algebraic specification literature.
A pair $(\Delta,\Phi)$ consisting of a signature $\Delta$ and a set of sentences $\Phi$ over $\Delta$ is called a \emph{presentation}.
We let $\Phi^\bullet$ denote $\{\varphi\in\Sen(\Delta)\mid \Phi\models \varphi\}$, the closure of $\Phi$ under the global consequence relation.
A \emph{presentation morphism} $\chi:(\Delta,\Phi)\to (\Delta^1,\Phi^1)$ consists of a signature morphism $\chi:\Delta\to\Delta^1$ such that $\Phi^1\models\chi(\Phi)$.
Any presentation $(\Delta,T)$ such that $T=T^\bullet$ is called a \emph{theory} .
A \emph{theory morphism} is just a presentation morphism between theories.

\paragraph{Examples}
Fragments of $\HFOL$ have been studied extensively in the literature.
We give a few examples.

\begin{example}(Rigid First-Order Hybrid Logic ($\RFOHL$) \cite{DBLP:conf/wollic/BlackburnMMH19})\label{ex:RFOHL}
This logic is obtained from $\HFOL$ by restricting the signatures $\Delta=(\Sigma^\nom,\Sigma^\rigid\subseteq\Sigma)$ such that 
(a)~$\Sigma^\nom$ has only one binary modality, 
(b)~$\Sigma$ is single-sorted, 
(c)~the unique sort is rigid,
(d)~there are no rigid function symbols except variables (regarded here as special constants), and 
(e)~there are no rigid relation symbols.
\end{example}
\begin{example}(Hybrid First-Order Logic with user-defined Sharing ($\HFOLS$)) \label{ex:HFOLS}
This logic has the same signatures and Kripke structures as $\HFOL$.
The sentences are obtained from atoms constructed with open terms only, that is,
if $\Delta=(\Sigma^\nom,\Sigma^\rigid\subseteq\Sigma)$, 
all (ground) equations over $\Delta$ are of the form $t_1=t_2$, where $t_1,t_2\in T_\Sigma$, and
all (ground) relation over $\Delta$ are of the form $\varpi(t)$, where $(\varpi:\ari)\in P$ and $t\in T_{\Sigma,\ari}$.
A version of $\HFOLS$ is the underlying logic of H system~\cite{cod-h}.
Other variants of $\HFOLS$ have been studied in \cite{martins,dia-msc,dia-qvh}. 
\end{example}
\begin{example}(Hybrid Propositional Logic ($\HPL$)) \label{ex:HDPL}
This is the most common form of multi-modal hybrid logic (e.g. \cite{ArecesB01}).
$\HPL$ is obtained from $\HFOL$ by restricting the signatures $\Delta=(\Sigma^\nom,\Sigma^\rigid\subseteq\Sigma)$ such that $\Sigma^\rigid$ is empty and the set of sorts in $\Sigma$ is empty.
Notice that if $\Sigma=(S,F,P)$ and $S=\emptyset$ then $P$ contains only propositional symbols.
\end{example}

\paragraph{Reachability}
Let $(W,M)$ be a Kripke structure over a signature $\Delta$.
\begin{itemize}
\item A possible world $w\in|W|$ is called \emph{reachable} if it is the denotation of some nominal,
that is, $w= W_k$ for some nominal $k\in F^\nom$.
\item Let $w\in|W|$ be a possible world and $s\in S$ a sort. 
An element $e\in M_{w,s}$ is called reachable if it is the denotation of some hybrid term, that is, $w=W_k$ and $e=M_{w,t}$ for some nominal $k\in F^\nom$ and rigid hybrid term $t\in T_{@\Sigma,@_k s}$.
\item $(W,M)$ is reachable by an $S^\ext$-sorted set $C$ of nominals and rigid hybrid terms if 
(a)~its set of possible worlds consists of denotations of nominals in $C$, and 
(b)~its carrier sets for the rigid sorts consist of denotations of rigid hybrid terms from $C$.
\item $(W,M)$ is reachable if $(W,M)$ is reachable by nominals and rigid hybrid terms.
\end{itemize}
The notion of reachability is connected to quantification, which is the reason for considering a Kripke structure reachable if its elements of rigid sorts are denotations of terms, thus disregarding elements of flexible sorts.
This notion is semantic and it makes sense also for fragments of $\HFOL$ whose sentences do not contain rigid hybrid terms such as $\RFOHL$, $\HFOLS$ or $\HPL$.
In institution theory, the notion of reachability was originally defined in \cite{Petria07} at an abstract level, and it played an important role in proving several proof-theoretic results~\cite{gai-com,gai-cbl,gai-bir} as well as model-theoretic properties~\cite{gai-int,gai-dls,gai-her,tutu-iilp}.

\section{Basic definitions and results} \label{sec:basics}
In this section, we establish the terminology and we state some foundational results necessary for the present study.
We start by noticing that rigid quantification cannot refer to unreachable elements of flexible sorts.
\begin{lemma} \label{lemma:reach-equiv}
Let $(W,M)$  be a reachable Kripke structure over a signature $\Delta$.
Let $(W,N)$ be a Kripke structure obtained from $(W,M)$ by
(a)~replacing unreachable elements of flexible sorts by some new elements, 
(b)~preserving the interpretation of function and relation symbols on the elements inherited from $(W,M)$, and
(c)~interpreting function symbols arbitrarily on the new arguments.
Then $(W,M)$ and $(W,N)$ are elementarily equivalent, in symbols,  $(W,M)\equiv(W,N)$ (that is, $(W,M)\models\varphi$ iff $(W,N)\models\varphi$, for all $\varphi\in\Sen(\Delta)$). 
\end{lemma}
The proof of the lemma above is straightforward by induction on the structure of sentences.
We recall Robinson consistency property as stated in institution theory (see, for example, ~\cite{DBLP:journals/sLogica/GainaP07}).
\begin{definition} \label{def:rob}
Consider the following square $\mathcal{S}$ of signature morphisms.
\begin{center}
\begin{tikzcd}
 \Delta^2 \ar[r,"\upsilon_2"] & \Delta'\\
 \Delta \ar[u,"\chi_2"] \ar[r,swap,"\chi_1"] & \Delta^1 \ar[u,swap,"\upsilon_1"]
\end{tikzcd}
\end{center}
$\mathcal{S}$ is a \emph{Robinson square}, if for every consistent theories $T^1 \subseteq \Sen(\Delta^1)$, $T^2 \subseteq \Sen(\Delta^2)$ and complete theory $T\subseteq \Sen(\Delta)$ such that $\chi_1$, $\chi_2$ are theory morphisms, it holds that $\upsilon_1(T^1) \cup \upsilon_2(T^2)$ is consistent.
\end{definition}
As it was shown in~\cite{gai-godel}, $\HFOL$ is compact, which means that interpolation is equivalent to Robinson consistency property.
\begin{proposition}
The following are equivalent for a commutative square $\mathcal{S}$ of signature morphisms as depicted in the diagram of Definition~\ref{def:rob}:
\begin{enumerate}[1)]
\item $\mathcal{S}$ is a Robinson square.
\item For every consistent theories $T^1 \subseteq \Sen(\Delta^1)$ and $T^2 \subseteq \Sen(\Delta^2)$ such that $\chi_1^{-1}(T^1)\cup\chi_2^{-1}(T^2)$ is consistent, the set $\upsilon_1(T^1)\cup \upsilon_2(T^2)$ is consistent.
\item $\mathcal{S}$ is a Craig Interpolation (CI) square , that is, for every $\Phi^1\subseteq\Sen(\Delta^1)$ and $\Phi^2\subseteq \Sen(\Delta^2)$ such that $\upsilon_1(\Phi^1)\models \upsilon_2(\Phi^2)$ there exists $\Phi\subseteq \Sen(\Delta)$ such that $\Phi^1\models \chi_1(\Phi)$ and $\chi_2(\Phi)\models\Phi^2$.
\end{enumerate}
\end{proposition}
The equivalence of the first two statements can be proved similarly to \cite[Proposition 6]{DBLP:journals/sLogica/GainaP07}, while the equivalence of first and last statement can be shown using ideas from~\cite[Corollary 3.1]{tar-bit}.

\begin{example} \label{ex:counter-1}
Let
$\Delta^1\stackrel{\chi_1}\leftarrow \Delta\stackrel{\chi_2}\to\Delta^2$ 
be a span of signature morphisms such that
\begin{itemize}
\item 
(a)~$\Delta$ has three nominals $\{k_1,k_2,k_3\}$, 
three flexible sorts $\{s_1,s_2,s_3\}$ and 
three flexible constants $\{c_1:\to s_1,c_2:\to s_2,c_3:\to s_3\}$;
(b)~$\Delta^1$ has two nominals $\{k,k_3\}$, 
one flexible sort $\{s\}$ and 
two flexible constants $\{c:\to s,c_3:\to s\}$;
(c)~$\Delta^2$ has two nominals $\{k,k_1\}$, 
two flexible sorts $\{s,s_2\}$ and 
three flexible constants $\{c_1:\to s,c_2:\to s_2,c_3:\to s\}$;
\item 
(a)~on nominals $\chi_1(k_1)=\chi_1(k_2)=k$, $\chi_1(k_3)=k_3$, 
on sorts $\chi_1(s_1)=\chi_1(s_2)=\chi_1(s_3)=s$, 
on function symbols $\chi_1(c_1:\to s_1)=\chi(c_2:\to s_2)=c:\to s$, $\chi_1(c_3:\to s_3)=c_3:\to s$; 
(b)~on nominals $\chi_2(k_1)=k_1$, $\chi_2(k_2)=\chi_2(k_3)=k$,
on sorts $\chi_2(s_1)=\chi_2(s_3)=s$, $\chi_2(s_2)=s_2$, 
on function symbols
$\chi_2(c_1:\to s_1)=c_1:\to s$,
$\chi_2(c_2:\to s_2)=c_2:\to s_2$,
$\chi_2(c_3:\to s_3)=c_3:\to s$.
\end{itemize}
Let $\Delta^1\stackrel{\upsilon_1}\to \Delta' \stackrel{\upsilon_2}\leftarrow\Delta^2$ be a pushout of the above span such that
\begin{itemize}
\item $\Delta'$ has 
one nominal $\{k\}$,
one flexible sort $\{s\}$ and 
two flexible constants $\{c:\to s, c_3:\to s\}$;
\item 
(a)~$\upsilon_1(k)=\upsilon_1(k_3)=k$, $\upsilon_1(c:\to s)= c:\to s$ and $\upsilon_1(c_3:\to s)=c_3:\to s$;
(b)~$\upsilon_2(k)=\upsilon_2(k_1)=k$, $\upsilon_2(c_1:\to s)=c:\to s$, $\upsilon_2(c_2:\to s_2)=c:\to s$ and $\upsilon_2(c_3:\to s)=c_3:\to s$.
\end{itemize}
\end{example}

According to the following lemma, interpolation doesn't hold in $\HFOL$, in general.

\begin{lemma}\label{lemma:counter-1}
The pushout described in Example~\ref{ex:counter-1} is not a CI square.
\end{lemma}
\begin{proof}
Let 
$\Phi^1\coloneqq\{\at{k_3}(c=c_3)\}$ and 
$\Phi^2\coloneqq\{\at{k_1}(c_1 = c_3)\}$.
Obviously, $\upsilon_1(\Phi^1)\models \upsilon_2(\Phi^2)$.
Suppose towards a contradiction that there exists an interpolant $\Phi$ over $\Delta$ such that 
$\Phi^1\models \chi_1(\Phi)$ and $\chi_2(\Phi)\models \Phi^2$.

 Let $(W^1,M^1)$ be the Kripke structure over $\Delta^1$ defined as follows:
$W^1$ consists of one possible world $w$, and 
$M^1_{w}$ is the single-sorted algebra such that $M^1_{w,s}=\{d,e\}$ and $M^1_{w,c}=M^1_{w,c_3}=d$. 
We have $(W^1,M^1)\models\Phi^1$, and since $\Phi^1\models \chi_1(\Phi)$, 
we get $(W^1,M^1)\models\chi_1(\Phi)$.
By the satisfaction condition, $(W^1,M^1)\red_{\chi_1}\models \Phi$.
Let $(V,N)$ be the Kripke structure over $\Delta$ obtained from $(W^1,M^1)\red_{\chi_1}$ by 
changing the interpretation of $c_3:\to s_3$ from $d$ to $e$, which implies that $V_{w,c_3}=e$. 
There exists an isomorphism $h:(V,N)\to (W^1,M^1)\red_{\chi_1}$ such that $h_{w,s_1}$ and $h_{w,s_2}$ are identities, while $h_{w,s_3}(d)=e$ and $h_{w,s_3}(e)=d$.
It follows that $(V,N)\models\Phi$.
There exists a $\chi_2$-expansion $(V^2,N^2)$ of $(V,N)$.
By the satisfaction condition, $(V^2,N^2)\models \chi_2(\Phi)$.
Since $N^2_{w,c_1}=N_{w,c_1}=d$ and $N^2_{w,c_3}=N_{w,c_3}=e$, 
we have $(V^2,N^2)\not \models \Phi^2$, contradicting $\chi_2(\Phi)\models \Phi^2$.
\end{proof}
We are interested in characterizing a span of signature morphisms whose pushout is a CI square.
For this purpose, it is necessary to restrict one of the arrows of the underlying span according to the following definition.

\begin{definition} \label{def:flex}
A signature morphism $\chi:\Delta\to \Delta^1$ \emph{preserves flexible symbols}~if
\begin{enumerate}[1)]
\item $\chi$ preserves flexible sorts, that is, $\chi(s)\in S_1^\flex$ for all $s\in S^\flex$, and
\item $\chi$ adds no new flexible operations on `old' flexible sorts, that is, for all flexible sorts $s\in S^\flex$ and all function symbols $\sigma_1:\ari_1\to \chi(s)\in F_1^\flex$ there exists $\sigma:\ari\to s\in F^\flex$ such that $\chi(\sigma:\ari\to s)=\sigma_1:\ari_1\to\chi(s)$.
\end{enumerate}
If, in addition, $\chi$ is injective on flexible sorts and on flexible function and relation symbols that have at least one flexible sort $s\in S^\flex$ in the arity then we say that $\chi$ \emph{protects flexible symbols}.
\end{definition}

 If $\chi:\Delta\to \Delta^1$ is an inclusion that preserves flexible sorts and adds no new function symbols $\sigma:\ari\to s$ with $s\in S^\flex$ on $\Delta$
then $\chi:\Delta\to \Delta^1$ protects flexible symbols.
If $\Delta$ has no flexible sorts then $\chi:\Delta\to\Delta^1$ protects flexible symbols.
In particular, if $\Delta$ is a $\HPL$ or $\RFOHL$ signature then $S^\flex=\emptyset$, which means that $\chi$ protects flexible symbols.
In applications, $\chi:\Delta\to\Delta^1$ from Definition~\ref{def:flex} is appropriate for hiding information,
which makes $\HFOL$ an instance of the abstract completeness result for structured specifications proved in \cite{DBLP:journals/tcs/Borzyszkowski02}.

\begin{lemma}[Lifting Lemma] \label{lemma:lifting}
Consider the following:
\begin{enumerate}[1)]
\item a signature morphism $\chi:\Delta\to \Delta^1$ which is injective on sorts and nominals,  and protects flexible symbols;
\item a set $C^1$ of new nominals and new rigid constants for $\Delta^1$;
\item $(V^1,N^1)\in|\Mod(\Delta^1(C^1))|$ reachable by $C^1$ and  
$(W,M)\in|\Mod(\Delta(C))|$ reachable by $C$ such that $(V^1,N^1)\red_{\chi^C}\equiv (W,M)$, where
\begin{itemize}
\item $C$ is the reduct of $C^1$ across $\chi$, i.e.
$C\coloneqq\{c:\to s\mid c:\to \chi(s)\in C^1\}$, and
\item $\chi^C:\Delta(C)\to\Delta^1(C^1)$ is the extension of $\chi$ that maps each constant $c:\to s\in C$ to $c:\to \chi(s)\in C^1$.
\end{itemize}
\end{enumerate}
Then $(W^1,M^1)\equiv(V^1,N^1)$ for some $\chi^C$-expansion $(W^1,M^1)$ of $(W,M)$.
\end{lemma}

\begin{proof}
Let $(V^1,N^1)\in|\Mod(\Delta^1(C^1))|$ and
$(W,M)\in|\Mod(\Delta(C))|$ be Kripke structures such that
$(V^1,N^1)\red_\chi\equiv (W,M)$. To keep notation consistent, we let
$(V,N)$ be the reduct  $(V^1,N^1)\red_\chi$ of $(V^1,N^1)$, and we will
construct an expansion $(W^1,M^1)$ of $(W,M)$ such that $(W^1,M^1)\equiv(V^1,N^1)$ in three steps.

\begin{enumerate}[1)]
\item We construct an isomorphism $h:(W,M)\to (V,R)$, 
where $(V,R)$ is obtained from $(V,N)$ by replacing all unreachable elements by the unreachable elements from $(W,M)$.

Firstly, we define $h$ as a function, which implicitly means that we define the universe of $(V,R)$.
Let $k\in C_\nom$ be a nominal, $v\coloneqq V_k$ and $w\coloneqq W_k$. 

\noindent\textsc{Case $s\in S^\rigid$:} We define the set $R_{v,s}\coloneqq N_{v,s}$, and 
the function $h_{w,s}:M_{w,s}\to R_{v,s}$ by $h_{w,s}(M_{w,c})=N_{v,c}$ for all constants $c:\to s\in C$.
Since both $(W,M)$ and $(V,N)$ are reachable by $C$ and $(W,M)\equiv(V,N)$, the function $h_{w,s}:M_{w,s}\to R_{v,s}$ is bijective.

\noindent\textsc{Case} $s\in S^\flex$:
Let $R_{v,s}$ be the set obtained from $N_{v,s}$ by removing all unreachable elements and adding all unreachable elements from $M_{w,s}$.
We define $h_{w,s}:M_{w,s}\to R_{v,s}$ by $h_{w,s}(M_{w,t})=N_{v,t}$ 
for all rigid hybrid $\Delta(C)$-terms $t$ of sort $@_k s$, and
$h_{w,s}(e)=e$ for all unreachable elements $e\in M_{w,s}$.
Since $(W,M)\equiv(V,N)$, the function $h_{w,s}:M_{w,s}\to R_{v,s}$ is bijective. 

Secondly, we interpret the function and relation symbols from $\Delta(C)$ in $(V,R)$.
Let $k\in C_\nom$ be a nominal, $v\coloneqq V_k$ and $w\coloneqq W_k$.

\noindent\textsc{Case} $\sigma:\ari\to s\in F(C)$:
We define $R_{v,\sigma}:R_{v,\ari}\to R_{v,s}$ by $R_{v,\sigma}(e)=h_{w,s}(M_{w,\sigma}(h_{w,\ari}^{-1}(e)))$ for all elements $e\in R_{v,\ari}$.
Since $(W,M)\equiv(V,N)$, we have $R_{v,\sigma} (e) = N_{v,\sigma}(e)$ for all reachable elements $e\in N_{v,\ari}\cap R_{v,\ari}$.

\noindent\textsc{Case $\varpi:\ari\in P$}
We define $R_{v,\varpi}\coloneqq h_{w,\ari}(M_{w,\varpi})$.
Since $(W,M)\equiv(V,N)$, we have $e\in R_{v,\varpi}$ iff $e\in N_{v,\varpi}$ for all reachable elements $e\in N_{v,\ari}\cap R_{v,\ari}$. 
By construction, $h:(W,M)\to(V,R)$ is a homomorphism, and since it is bijective, $h:(W,M)\to(V,R)$ is an isomorphism.
\item We define an expansion $(V^1,R^1)$ of $(V,R)$ along $\chi$ such that $(V^1,R^1)\equiv(V^1,N^1)$.
Roughly, $(V^1,R^1)$ is obtained from $(V^1,N^1)$ by replacing all unreachable elements of sorts in $\chi(S^\flex)$ with unreachable elements of flexible sorts from $(V,R)$.
Concretely, $(V^1,R^1)$ is obtained from $(V^1,N^1)$ as follows:

\noindent\textsc{Case} $s_1\in \chi(S^\flex)$:
$R^1_{v,s_1}\coloneqq R_{v,\chi^{-1}(s_1)}$ for all $v\in|V^1|$,  which is well-defined since $\chi$ is injective on sorts.

\noindent\textsc{Case} $\sigma_1:\ari_1\to s_1\in \chi(F^\flex)$, where $\ari_1$ contains at least one sort from $\chi(S^\flex)$:
For all $v\in|V^1|$, $R^1_{v,\sigma_1}\coloneqq R_{v,\chi^{-1}(\sigma_1)}$.
Since $\chi$ protects flexible symbols, $\chi^{-1}(\sigma_1)$ is unique, which means that $R^1_{v,\sigma_1}$ is well-defined.
Also, we have $N^1_{v,\sigma_1}(e)=N_{v,\chi^{-1}(\sigma_1)}(e)=R_{v,\chi^{-1}(\sigma_1)}(e)$ for all reachable elements $e\in N^1_{v,\ari_1}$.

\noindent\textsc{Case} $\pi_1:\ari_1\in \chi(P^\flex)$, where $\ari_1$ contains at least one sort from $\chi(S^\flex)$:
For all $v\in|V^1|$, $R^1_{v,\pi_1}\coloneqq R_{v,\chi^{-1}(\pi_1)}$.
Since $\chi$ protects flexible symbols, $\chi^{-1}(\pi_1)$ is unique, which means that $R^1_{v,\pi_1}$ is well-defined.
Also, we have $e\in N^1_{v,\pi_1}=N_{v,\chi^{-1}(\pi_1)}$ iff $e\in R_{v,\chi^{-1}(\pi_1)}=R^1_{v,\pi_1}$ for all reachable elements $e\in N^1_{v,\ari_1}$.

\noindent\textsc{Case} $\sigma_1:\ari_1\to s_1\in F_1^\flex\setminus \chi(F^\flex)$, where $\ari_1$ has at least one sort from $\chi(S^\flex)$:
For all $v\in|V^1|$, 
the function $R^1_{v,\sigma_1}: R^1_{v,\ari_1}\to R^1_{v,s_1}$ is defined by 
\begin{itemize}
\item $R^1_{v,\sigma_1}(e)=N^1_{v,\sigma_1}(e)$ for all elements $e\in R^1_{v,\ari_1}\cap N^1_{v,\ari_1}$ and
\item $R^1_{v,\sigma_1}(e)$ is an arbitrary value in $R^1_{v,s_1}$ for all unreachable $e\in R^1_{v,\ari_1}\setminus N^1_{v,\ari_1}$.
\end{itemize}

\noindent\textsc{Case} $\pi_1:\ari_1\in P_1^\flex\setminus \chi(P^\flex)$, where $\ari_1$ contains at least one sort from $\chi(S^\flex)$: 
For all possible worlds $v\in|V^1|$, 
let $R^1_{v,\pi_1}\coloneqq\{e\in N^1_{v,\pi_1}\mid e \text{ is reachable}\}$. 
Now, since $\chi$ protects flexible symbols, $\chi$ preserves flexible symbols, which means that for all possible worlds $v\in |V|$ and all sorts $s\in S$,
\begin{center} 
$e\in N_{v,s}$ is unreachable iff $e\in N^1_{v,\chi(s)}$ is unreachable.
\end{center}
It follows that the reachable sub-structures of $(V^1,N^1)$ and $(V^1,R^1)$ coincide.
By Lemma~\ref{lemma:reach-equiv}, $(V^1,N^1)\equiv(V^1,R^1)$.
\item We define an isomorphism $h^1:(W^1,M^1)\to(V^1,R^1)$ by expanding $h:(W,M)\to (V,R)$ along $\chi$.

Firstly, we define $h^1$ as a function.
Let $h^1:W\to V$ be $h:W\to V$, which is  bijective.
Assume $k$ is a nominal in $C$ and let $w=W_k$.

\noindent\textsc{Case} $s_1\in \chi(S)$: 
Let $h^1_{w,s_1}\coloneqq h_{w,\chi^{-1}(s_1)}$, which is bijective.

\noindent\textsc{Case} $s_1\in S^1\setminus \chi(S)$:
$M^1_{w,s_1}\coloneqq R^1_{w,s_1}$ and $h^1_{w,s_1}:M^1_{w,s_1}\to R^1_{w,s_1}$ is the identity.

Secondly, we interpret the function and relation symbols from $\Delta^1(C^1)$ in $(W^1,M^1)$.
For any nominal or modality $x$ in $\Delta^1(C^1)$, we define  $W^1_x\coloneqq (h^1)^{-1}(V^1_x)$.
Take a nominal $k$ in $C^1$, and let $w\coloneqq W^1_k$ and $v\coloneqq V^1_k$.

\noindent\textsc{Case} $\sigma_1:\ari_1\to s_1\in F^1(C^1)$:
We define $M^1_{w,\sigma_1}:M^1_{w,\ari_1}\to M^1_{w,s_1}$ by 
$M^1_{w,\sigma_1}(e)=(h^1_{w,s_1})^{-1}(R^1_{v,\sigma_1}(h^1_{w,\ari_1}(e)))$ 
for all elements $e\in M^1_{w,\ari_1}$.

\noindent\textsc{Case} $\pi_1:\ari_1\in P^1$:
We define $M^1_{w,\pi_1}\coloneqq (h^1_{w,\ari_1})^{-1}(R^1_{\pi_1})$.\\
Since $h:(W,M)\to (V,R)$ is an isomorphism, $h^1:(W^1,M^1)\to (V^1,R^1)$ is an isomorphism too.
\end{enumerate}
It follows that $(W^1,M^1)\equiv(V^1,R^1)$. 
Since $(V^1,R^1)\equiv(V^1,N^1)$, we get $(W^1,M^1)\equiv(V^1,N^1)$.
\end{proof}

Definition~\ref{def:flex} provides a general criterion for proving Robinson consistency property, 
while Lemma~\ref{lemma:lifting} is essential for completing the proof of Robinson consistency theorem.

\section{Relativization} \label{sec:relat}
Relativization is a well-known method in classical model theory for defining substructures and their properties~\cite{DBLP:books/daglib/0080659}. 
The substructures are usually characterized by some unary predicate and the technique is necessary in the absence of sorts when dealing with modular properties (such as putting together models defined over different signatures) which implicitly involve signature morphisms.
For $\HFOL$, the relativization is necessary to prove Robinson consistency property from omitting types property, since the signatures of nominals are single-sorted.
It is worth mentioning that relativization is not necessary to prove Robinson consistency for many-sorted first-order logic.
\begin{definition} \label{def:amalg}
The \emph{relativized union} of any signatures $\Delta^1$ and $\Delta^2$ is a presentation $(\Delta^\diamond,\Phi^\diamond)$ defined as follows:
\begin{enumerate}[1)]
\item $\Delta^\diamond$ is the signature obtained from $\Delta^1\coprod\Delta^2$ by adding 
two nominals $o_1$ and $o_2$, and
two unary modalities ${\pi_1:1}$ and ${\pi_2:1}$,

\item $\Phi^\diamond\subseteq\Sen(\Delta^\diamond)$ consists of $\pi_1\vee\pi_2$ and all sentences of the form $\at{k_i}\pi_i$, where $i\in\{1,2\}$ and  $k_i\in F^\nom_i\cup\{o_i\}$.
\end{enumerate}
\end{definition}

Let $\inj_i:\Delta^i\to \Delta^1\coprod\Delta^2$ be the canonical injection, for each $i\in\{1,2\}$.
Let $\theta:\Delta^1\coprod\Delta^2\hookrightarrow \Delta^\diamond$ be an inclusion.
Here, we are interested more in the vertex $\Delta^\diamond$ and less in the arrows $(\inj_i;\theta)$, where $i\in\{1,2\}$.
The presentation $(\Delta^\diamond,\Phi^\diamond)$ is meant to define Kripke structures obtained from the union of a Kripke structure over $\Delta^1$ and a Kripke structure over $\Delta^2$.
The new nominals $o_1$ and $o_2$ together with the sentences $\at{o_1} \pi_1$ and $\at{o_2}\pi_2$ ensure that the domains of $\pi_1$ and $\pi_2$ are not empty. 
For each nominal $k_1\in F_1^\nom$, the sentence $\at{k_1}\pi_1$ ensures that the interpretation of $k_1$ belongs to the denotation of $\pi_1$. 
A similar remark holds for any sentence $\at{k_2}\pi_2$ with $k_2\in F_2^\nom$.
For the sake of simplifying the notation, we assume without loss of generality that $\Delta^1$ and $\Delta^1$ are disjoint, which means that $\Delta^1\coprod \Delta^2=\Delta^1\cup \Delta^2$.

\begin{definition}
Let $\Delta^1$ and $\Delta^2$ be two disjoint signatures.
For each $i\in\{1,2\}$, the \emph{relativized reduct} $\red_{\pi_i}:\Mod(\Delta^\diamond)\to\Mod(\Delta^i)$ is defined as follows:
\begin{enumerate}[1)]
\item For each $(W,M)\in|\Mod(\Delta^\diamond,\Phi^\diamond)|$, the Kripke structure $(W,M)\red_{\pi_i}$ denoted $(W^i,M^i)$ is defined by
(a)~$|W^i|=  W_{\pi_i}$, 
(b)~$W^i_k =W_k$ for all nominals $k\in F^\nom_i$, 
(c)~$W^i_\varrho=W_\varrho\cap  W_{\pi_i}$ for all unary modalities $(\varrho:\nom)\in P_i^\nom$,
(d)~$W^i_\lambda=\{(w,v)\in W_\lambda \mid  w,v\in  W_{\pi_i} \}$ for all modalities $(\lambda:\nom~\nom)\in P^\nom_i$,
(e)~$M^i= M|_{W^i}$ and $M^i_{w,x}=M_{w,x}$ for all possible worlds $w\in |W^i|$ and all sort/function/relation symbols $x$ in $\Sigma_i$.

\item For each $h:(W,M)\to (W',M')\in \Mod(\Delta^\diamond,\Phi^\diamond)$, the homomorphism $h\red_{\pi_i}:(W,M)\red_{\pi_i}\to (W',M')\red_{\pi_i}$ is defined by $(h\red_{\pi_i})_w=h_w$ for all $w\in  W_{\pi_i}$.
\end{enumerate}
\end{definition}

\begin{definition}
The \emph{relativized translation} 
$\rt(\pi_i):\Sen(\Delta^i)\to \Sen(\Delta^\diamond,\Phi^\diamond)$, 
where $i\in\{1,2\}$, is defined by induction on the structure of sentences, simultaneously, for all disjoint signatures $\Delta^1$ and $\Delta^2$:
\begin{enumerate}[1)]
\item $\rt(\pi_i)(k)\coloneqq \pi_i\Rightarrow k$ for all nominals $k\in F^\nom_i$.
\item $\rt(\pi_i)(\varrho)\coloneqq \pi_i\Rightarrow \varrho$ for all unary modalities $(\varrho:\nom)\in P_i^\nom$.
\item $\rt(\pi_i)(t_1=t_2)\coloneqq \pi_i\Rightarrow t_1 = t_2$ for all ground equations $t_1=t_2$ over $\Delta^i$.
\item $\rt(\pi_i)(\varpi(t_1,\dots,t_n))\coloneqq \pi_i\Rightarrow \varpi(t_1,\dots t_n)$ for all ground relations $\varpi(t_1,\dots,t_n)$ over $\Delta^i$.
\item $\rt(\pi_i)(\at{k}\gamma)\coloneqq \pi_i\Rightarrow \at{k}\rt(\pi_i)(\gamma)$ for all nominals $k\in F^\nom_i$ and all sentences $\gamma\in\Sen(\Delta^i)$. 
\item $\rt(\pi_i)(\pos{\lambda}\gamma)\coloneqq \pi_i\Rightarrow \pos{\lambda}\pi_i\wedge \rt(\pi_i)(\gamma)$ for all modalities $(\lambda:\nom~\nom) \in P_i^\nom$ and all sentences $\gamma\in \Sen(\Delta^i)$.
\item $\rt(\pi_i)(\neg \gamma) \coloneqq \pi_i \Rightarrow \neg\rt(\pi_i)(\gamma)$ for all $\gamma\in\Sen(\Delta^i)$.

%
\item $\rt(\pi_i)(\gamma_1\vee\gamma_2)\coloneqq \rt(\pi_i)(\gamma_1)\vee \rt(\pi_i)(\gamma_2)$ for all $\gamma_1,\gamma_2\in\Sen(\Delta^i)$.

\item \label{it:rt-10} 
$\rt(\pi_i)(\store{z}\gamma)\coloneqq \pi_i\Rightarrow \store{z} \rt(\pi_i)(\gamma)$ for all sentences $\store{z}\gamma\in \Sen(\Delta^i)$, where $z$ a nominal variable.
\footnote{Notice that $\Delta^\diamond(z)$ is obtained from the relativized union of $\Delta^i(z)$ and $\Delta^j$, where $i,j\in\{1,2\}$ and $i\neq j$; therefore, $\rt(\pi)(\gamma)$ is well-defined.}

%
%
%
%

%
\item $\rt(\pi_i)(\Exists{x}\gamma)\coloneqq \pi_i\Rightarrow \Exists{x} \at{x} \pi_i \wedge \rt(\pi_i)(\gamma)$ for all  sentences $\Exists{x}\gamma\in \Sen(\Delta^i)$ with $x$ a nominal variable.
\item $\rt(\pi_i)(\Exists{y}\gamma) \coloneqq \pi_i\Rightarrow \Exists{y} \rt(\pi_i)(\gamma)$ for all  sentences $\Exists{y}\gamma\in \Sen(\Delta^i)$ with $y$ a rigid variable.
\end{enumerate}
\end{definition}

Simultaneous induction for all disjoint signatures is necessary for the case corresponding to quantified sentences.
Unlike first-order logic, relativization is applied not only to quantifiers but it starts with atomic sentences. 
One can notice that locally the antecedent $\pi_i$ of the implication is redundant, but globally it is not.

\begin{proposition} [Satisfaction condition] \label{prop:sat-cond}
For all disjoint signatures $\Delta^1$ and $\Delta^2$,
all Kripke structures $(W,M)\in|\Mod(\Delta^\diamond,\Phi^\diamond)|$ and 
all sentences $\gamma\in\Sen(\Delta^i)$,
where $i\in\{1,2\}$, 
the following \emph{satisfaction conditions} hold:
\begin{itemize}
\item For all worlds $w\in  W_{\pi_i}$, 
we have $(W,M)\models^w \rt(\pi_i)(\gamma)$ iff $(W,M)\red_{\pi_i} \models^w \gamma$.

\item For all worlds $w\in|W|\setminus W_{\pi_i}$, 
we have $(W,M)\models^w\rt(\pi_i)(\gamma)$.

\item $(W,M)\models \rt(\pi_i)(\gamma)$ iff $(W,M)\red_{\pi_i}\models \gamma$.
\end{itemize}
\end{proposition}

The first two statements of Proposition~\ref{prop:sat-cond} are straightforward by induction on the structure of sentences. 
The third statement, which corresponds to the global satisfaction condition, is a consequence of the first two statements.

\begin{proposition} \label{prop:amalg}
For all disjoint signatures $\Delta^1$ and $\Delta^2$, and 
all Kripke structures $(W^1,M^1)$ and $(W^2,M^2)$ over $\Delta^1$ and $\Delta^2$, respectively, 
there exists $(W,M)\in|\Mod(\Delta^\diamond,\Phi^\diamond)|$ called the \emph{relativized union} of $(W^1,M^1)$ and $(W^2,M^2)$ such that $(W,M)\red_{\pi_1}=(W^1,M^1)$ and $(W,M)\red_{\pi_2}=(W^2,M^2)$.
\end{proposition}

The relativized union of Kripke structure is not unique.

\section{Robinson consistency} \label{sec:rc}
Robinson consistency property is derived from Omitting Types Theorem, which was proved in \cite{gai-hott} for $\HFOL$.
We start by defining the semantic opposite of a sentence.
In first-order logic the semantic opposite of a sentence is its negation.
\begin{definition}
Given a sentence $\psi$ over a signature $\Delta$, we let 
(a)~$+\psi$ denote the sentence $\Forall{z^\circ}\at{z^\circ}\psi$,
(b)~$-\psi$ denote the sentence $\Exists{z^\circ}\at{z^\circ}\neg\psi$, and
(c)~$\pm\psi$ range over $\{+\psi,-\psi\}$,
where $z^\circ$ is a distinguished nominal variable for $\Delta$.
\end{definition}

The proof of the following lemma is straightforward.

\begin{lemma} \label{lemma:+}
For all Kripke structures $(W,M)$ and all sentences $\psi$ over a signature $\Delta$, we have 
(a)~$(W,M) \models \psi$ iff $(W,M)\models +\psi$ iff $(W,M)\models^w+\psi$ for some possible world $w\in|W|$, and
(b)~$(W,M)\not\models \psi$ iff $(W,M)\models -\psi$.
\end{lemma}

By Lemma~\ref{lemma:+}, the satisfaction of $+\psi$ does not depend on the  possible world where the sentence $+\psi$ is evaluated. The same comment holds for $-\psi$ too.

\subsection{Framework}
We set the framework in which Robinson consistency property is proved.
Let $(\Delta^1,\Phi^1)\stackrel{\chi_1}\leftarrow(\Delta,\Phi)\stackrel{\chi_2}\to(\Delta^2,\Phi^2)$ be a span of presentation morphisms such that
(a)~$\chi_2$ is injective on sorts and nominals, and protects flexible sorts,
(b)~$\Phi$ is maximally consistent over $\Delta$, and 
(c)~$\Phi^i$ is consistent over $\Delta^i$ for each $i\in\{1,2\}$.

Assume a set of new rigid constants $C$ for $\Delta$ such that 
$\card(C_s)=\alpha$ for all rigid sorts $s\in S^\rigid$, where $\alpha\coloneqq max\{\card(\Sen(\Delta^1)),\card(\Sen(\Delta^2))\}$.
For each $i\in\{1,2\}$, 
let $C^i$ be the set of new constants for $\Delta^i$ obtained 
by renaming the translation of the constants in $C$ along $\chi_i$ and
by adding a set of new constants $C^i_{s^i}$ of cardinality $\alpha$ for each rigid sort $s^i\in S^\rigid_i$ outside the image of $\chi_i$:
\begin{center}
$C^i\coloneqq \{c^i:\to\chi_i(s) \mid c:\to s\in C\} \cup (\bigcup_{s^i\in S^\rigid_i \setminus \chi_i(S^\rigid)} C^i_{s^i})$
\end{center}
where
(a)~each constant $c^i:\to\chi_i(s)$ is the renaming of a constant $c:\to s\in C$, and
(b)~$C^i_{s^i}$ is a set of new constants of sort $s^i$ for all rigid sorts $s^i\in S^\rigid_i \setminus \chi_i(S^\rigid)$.
Let $\chi_i^C:\Delta(C)\to\Delta^i(C^i)$ be the extension of $\chi_i:\Delta\to\Delta^i$ to $\Delta(C)$ which maps each constant $c:\to s \in C$ to its renaming $c^i:\to \chi_i(s)\in C^i$.
Without loss of generality we assume that $\Delta^1(C^1)$ and $\Delta^2(C^2)$ are disjoint.

\begin{figure}[h]
\centering
\begin{tikzcd}
\Delta(C) \ar[rr,"\chi_1^C"] \ar[rd,"\chi_2^C "] & & \Delta^1(C^1) \ar[rd,dotted,"(\_)^{\pi_1}"] & \\ 
& \Delta^2(C^2)  \ar[rr,dotted,"(\_)^{\pi_2}", near start] & & (\Delta^\diamond(C^\diamond),\Phi_C^\diamond) \\
& & & \\
\Delta \ar[uuu,hook] \ar[rr,"\chi_1" near start,dashed] \ar[rd,swap, "\chi_2"] & & \Delta^1 \ar[uuu,dashed,hook] \ar[dr,dotted,"(\_)^{\pi_1}"]&\\ 
 & \Delta^2 \ar[uuu,hook] \ar[rr,dotted,"(\_)^{\pi_2}"] & & (\Delta^\diamond,\Phi^\diamond) \ar[uuu,hook] \\
\end{tikzcd}
\caption{}
\label{fig:frame}
\end{figure}
In Figure~\ref{fig:frame},
$(\Delta^\diamond,\Phi^\diamond)$ is the relativized union of $\Delta^1$ and $\Delta^2$, 
while $(\Delta^\diamond(C^\diamond),\Phi_C^\diamond)$ is the relativized union of $\Delta^1(C^1)$ and $\Delta^2(C^2)$.
The definitions of $C$, $C^1$ and $C^2$ are unique up to isomorphism and they are essential for the proof of Robinson consistency theorem.

\begin{notation}[Semantics] 
For each $(W,M)\in|\Mod(\Delta^\diamond,\Phi^\diamond)|$, we let
\begin{itemize}
\item $(W^1,M^1)$ and $(W^2,M^2)$ denote $(W,M)\red_{\pi_1}$ and $(W,M)\red_{\pi_2}$, respectively;
\item $(W^a,M^a)$ and $(W^b,M^b)$ denote $(W^1,M^1){\red_{\chi_1}}$ and $(W^2,M^2){\red_{\chi_2}}$, respectively.
\end{itemize}
Moreover, let $(V,N)$ be a Kripke structure over $\Delta^\diamond(C^\diamond)$ and adopt a similar convention as above for its reducts.
\end{notation}

\begin{notation}[Syntax] Let $i\in\{1,2\}$ and $\psi\in\Sen(\Delta(C))$.

\begin{itemize}
\item Let $\psi^i$ and $\psi^{\pi_i}$ denote $\chi_i^C(\psi)$ and $\rt(\pi_i)(\chi_i^C(\psi))$, respectively. 
\item Let $\Phi^{\pi_i}$  denote $\rt(\pi_i)(\Phi^i)$.
\end{itemize}
\end{notation}

\subsection{Results}
This section contains the main results, which are Robinson consistency theorem and its corollaries. 
The following lemma is crucial for subsequent developments and it is a consequence of Proposition~\ref{prop:sat-cond} and Lemma~\ref{lemma:+}.

\begin{lemma}\label{lemma:bicond}
For all $(V,N)\in|\Mod(\Delta^\diamond(C^\diamond))|$ and all $\psi\in\Sen(\Delta(C))$,
\begin{center}
$(V,N)\models +(+\psi)^{\pi_1} \Leftrightarrow  +(+\psi)^{\pi_2} 
\text{ iff }
\Big((V^a,N^a)\models \psi \text{ iff }(V^b,N^b) \models \psi \Big)$.
\end{center}
\end{lemma}

Lemma~\ref{lemma:bicond} says that $(V,N)$ globally satisfies $+(+\psi)^{\pi_1} \Leftrightarrow  +(+\psi)^{\pi_2}$ for all $\Delta(D)$-sentences~$\psi$ iff
$(V^a,N^a)$ and $(V^b,N^b)$ are elementarily equivalent.

\begin{notation}
We define the following set of sentences over $\Delta^\diamond(C^\diamond)$: 
\begin{center} 
$T_C\coloneqq \Phi_C^\diamond \cup\Phi^{\pi_1} \cup \Phi^{\pi_2} \cup 
\{ +(+\psi)^{\pi_1} \Leftrightarrow  +(+\psi)^{\pi_2} \mid \psi\in\Sen(\Delta(C))\}$.
\end{center}
\end{notation}

Notice that $T_C$ describes Kripke structures $(V,N)$ over $\Delta^\diamond(C^\diamond)$ obtained from the relativized union of some Kripke structures $(V^1,N^1)\in|\Mod(\Delta^1(C^1),\Phi^1)|$ and $(V^2,N^2)\in|\Mod(\Delta^2(C^2),\Phi^2)|$ such that $(V^a,N^a)$ and $(V^b,N^b)$ are elementarily equivalent.

\begin{proposition} \label{prop:cons}
$T_C$ is consistent.
\end{proposition}

\begin{figure}[h]
\begin{tikzcd}
\pm\Psi \ar[r,dotted,no head]& \Delta(D) \ar[rr,"\chi_1^D"] \ar[rd,"\chi_2^D"] & & \Delta^1(D^1) \ar[rd,dotted,"\pi_1"] & \\ 
& & \Delta^2(D^2)  \ar[rr,dotted,"\pi_2",near start] & & (\Delta^\diamond(D^\diamond),\Phi_D^\diamond) \\
& & & & \\
\Exists{D}\bigwedge\pm\Psi \ar[r,dotted,no head] & \Delta \ar[uuu,hook] \ar[rr,"\chi_1" near start,dashed] \ar[rd,swap, "\chi_2"] & & \Delta^1 \ar[uuu,dashed,hook] \ar[dr,dotted,"\pi_1"]&\\ 
&  & \Delta^2 \ar[uuu,hook] \ar[rr,dotted,"\pi_2"] & & (\Delta^\diamond,\Phi^\diamond) \ar[uuu,hook] \\
\end{tikzcd}
\caption{}
\end{figure}

\begin{proof}
Let $\Psi$ be a finite set of $\Delta(C)$-sentences.
Let $D\subseteq C$ be the finite subset of all constants from $C$ that occur in $\Psi$.
We define $D^1\coloneqq \chi_1^C(D)$ and $D^2\coloneqq \chi_2^C(D)$.
For each $i\in\{1,2\}$, we let $\chi_i^D$ denote the restriction of $\chi_i^C$ to $\Delta(D)$.
We show that
$\Phi_D^\diamond \cup\Phi^{\pi_1} \cup \Phi^{\pi_2} \cup 
\{ + (+\psi)^{\pi_1} \Leftrightarrow  +(+\psi)^{\pi_2} \mid \psi\in\Psi\}$ is consistent, 
where $(\Delta^\diamond(D^\diamond),\Phi_D^\diamond)$ is the relativized union of $\Delta^1(D^1)$ and $\Delta^2(D^2)$.

\begin{itemize}
\item[] Since $\Phi^1$ is consistent,
$(W^1,M^1)\models \Phi^1$ for some Kripke structure $(W^1,M^1)$ over $\Delta^1$.
Let $(V^1,N^1)$ be an arbitrary expansion of $(W^1,M^1)$ to $\Delta^1(D^1)$.
Let $\pm \Psi\coloneqq \{ \pm \psi \mid \psi \in \Psi \text{ and } (V^a,N^a) \models \pm \psi \}$,
where $(V^a,N^a)=(V^1,N^1)\red_{\chi_1^D}$.
Since $(V^a,N^a)\models \pm\Psi$,
$(W^a,M^a)\models \Exists{D}\bigwedge\pm\Psi$.\footnote{Since $D$ is not a set of variables, 
$\Exists{D}\bigwedge\pm\Psi$ is not a sentence in our language,
but there exists a $\Delta$-sentence semantically equivalent to it.}
Since $\Phi^2$ is consistent,
$(W^2,M^2)\models \Phi^2$ for some Kripke structure $(W^2,M^2)$ over $\Delta^2$.
Since $\chi_1(\Phi)\subseteq \Phi^1$ and $\chi_2(\Phi)\subseteq \Phi^2$,
by satisfaction condition,
$(W^a,M^a)\models \Phi$ and $(W^b,M^b)\models \Phi$.
Since $\Phi$ is maximally consistent,
$(W^a,M^a)\equiv (W^b,M^b)$.
Since $(W^a,M^a)\models \Exists{D}\bigwedge\pm\Psi$,
$(W^b,M^b)\models \Exists{D}\bigwedge \pm \Psi$.
It follows that $(V^b,N^b)\models \pm\Psi$ for some expansion $(V^b,N^b)$ of $(W^b,M^b)$ to $\Delta(D)$.
Since $\{\Delta(D)\hookleftarrow \Delta\stackrel{\chi_2}\to\Delta^2,\Delta(D)\stackrel{\chi_2^D}\to \Delta^2(D^2)\hookleftarrow \Delta^2\}$ is a pushout and ${(V^b,N^b)\red_\Delta}=(W^b,M^b)={(W^2,M^2)\red_{\chi_2}}$,
there exists an expansion $(V^2,N^2)$ of $(W^2,M^2)$ to $\Delta^2(D^2)$ such that $(V^2,N^2)\red_{\chi_2^D}=(V^b,N^b)$.
Let $(V,N)$ be a relativized union of $(V^1,N^1)$ and $(V^2,N^2)$.
Since $(V^1,N^1)\models \Phi^1$ and $(V^2,N^2)\models \Phi^2$,
by Proposition~\ref{prop:sat-cond},
$(V,N)\models \Phi_D^\diamond\cup \Phi^{\pi_1}\cup\Phi^{\pi_2}$.
By Lemma~\ref{lemma:bicond},
$(V,N)\models \{+(+\psi)^{\pi_1} \Leftrightarrow +(+\psi)^{\pi_2}  \mid \psi\in\Psi\}$.
\end{itemize}
Hence, by compactness, $T_C$ is consistent.
\end{proof}

Recall that $\nom$ denotes the sort of nominals.

\begin{notation} [Nominal type]
We define a type in one nominal variable $z$:
\begin{center}
$\Gamma_\nom\coloneqq\{\at{z}\pi_i \Rightarrow z\neq c^i\mid i=\overline{1,2} \text{ and } c:\to \nom\in C\}$
\end{center}
where 
$c^i$ denotes $\chi_i(c)$ for all nominals $c:\to \nom\in C$, and
$z\neq c^i$ denotes $\neg\at{z}c^i$.
\end{notation}

Notice that a Kripke structure $(V,N)$ which satisfies $\pi_1\vee\pi_2$ and omits $\Gamma_\nom$ has the set of possible worlds reachable by the nominals in $C^\diamond$.

\begin{proposition} \label{prop:omit-n}
$T_C$ $\alpha$-omits $\Gamma_\nom$, that is, 
for each set of sentences $p\subseteq \Sen(\Delta^\diamond(C^\diamond,z))$ of cardinality strictly less than $\alpha$ such that $T_C\cup p$ is consistent, we have $T_C\cup p\not\models\Gamma_\nom$. 
\end{proposition}

\begin{proof}
Let $p\subseteq\Sen(\Delta^\diamond(C^\diamond,z))$ be a set of sentences of cardinality strictly less than $\alpha$ such that $T_C\cup p$ is consistent.
Let $D_\nom$ be the set of all nominals $c\in C_\nom$ such that either $c^1$ or $c^2$ occurs in $p$.
Since $\card(p)<\alpha$, we have $\card(D_\nom)<\alpha$. 
Let $D$ be the set of constants obtained from $C$ by removing all nominals from $C_\nom\setminus D_\nom$.
We define $D^i\coloneqq \chi_i^C(D)$ for each $i\in\{1,2\}$. 
It follows that $p\subseteq \Sen(\Delta^\diamond(D^\diamond,z))$, 
where $(\Delta^\diamond(D^\diamond),\Phi_D^\diamond)$ is the relativized union of $\Delta^1(D^1)$ and $\Delta^2(D^2)$.
Let $T_D$ be the set of all sentences from $T_C$ which contains only constants from $D^1$ and $D^2$.
Since $T_C\cup p$ is consistent, its subset $T_D\cup p$ is consistent too.
Let $(V,N)$ be a Kripke structure over $\Delta^\diamond(D^\diamond)$ such that $(V,N)\models T_D$ and let $v\in |V|$ such that $(V^{z\leftarrow v},N)\models p$.
Since $(V,N)\models \pi_1\vee \pi_2$, we have $v\in \pi_1^V$ or $v\in \pi_2^V$.
We assume that $v\in \pi_1^V$, as the case $v\in \pi_2^V$ is symmetrical.
According to our conventions,
$(V^1,N^1)= (V,N)\red_{\pi_1}$,
$(V^a,N^a)= (V,N)\red_{\chi_1^D}$,
$(V^2,N^2)= (V,N)\red_{\pi_2}$ and
$(V^b,N^b)= (V,N)\red_{\chi_2^D}$,
where $\chi_i^D$ denotes the restriction of $\chi_i^C$ to $\Delta(D)$ for each $i\in\{1,2\}$.
\begin{figure}[h]\small
\begin{tikzcd}
\pm\Psi \ar[r,dotted,no head]& \Delta(D,c) \ar[rr] \ar[rd] & & \Delta^1(D^1,c^1) \ar[rd,dotted,"\pi_1"] & \\ 
& & \Delta^2(D^2,c^2)  \ar[rr,dotted,"\pi_2",near start] & & \Delta^\diamond(D^\diamond,c^1,c^2) \\
& & & & \\
\Exists{c}\bigwedge\pm\Psi \ar[r,dotted,no head] & \Delta(D) \ar[uuu,hook] \ar[rr,"\chi_1^D" near start,dashed] \ar[rd,swap, "\chi_2^D"] & & \Delta^1(D^1) \ar[uuu,dashed,hook] \ar[dr,dotted,"\pi_1"]&\\ 
&  & \Delta^2(D^2) \ar[uuu,hook] \ar[rr,dotted,"\pi_2"] & & \Delta^\diamond(D^\diamond) \ar[uuu,hook] \\
\end{tikzcd}
\caption{}
\end{figure}

Since $\card(D_\nom)<\alpha=\card(C_\nom)$, there exists a nominal $c\in C_\nom\setminus D_\nom$.
Let $\Psi\subseteq \Sen(\Delta(D,c))\setminus\Sen(\Delta(D))$ be a finite set of sentences.
Then
$T_D \cup  p \cup \{(\at{c^1}\pi_1),(\at{c^2}\pi_2), (z = c^1) \} \cup \{ +(+\psi)^{\pi_1} \Leftrightarrow + (+\psi)^{\pi_2} \mid \psi\in\Psi\}$ is consistent:
\begin{itemize}
\item[] We define $\pm \Psi=\{ \pm\psi\mid ((V^a)^{c\leftarrow v},N^a) \models \pm \psi\}$.
Since $((V^a)^{c\leftarrow v},N^a) \models \pm\Psi$, 
$(V^a,N^a)\models \Exists{c}\bigwedge \pm\Psi$.
Since $(V,N)\models \{ +(+\varphi)^{\pi_1} \Leftrightarrow + (+\varphi)^{\pi_2} \mid \varphi\in\Sen(\Delta(D))\}$,
by Lemma~\ref{lemma:bicond},
$(V^a,N^a) \equiv (V^b,N^b)$.
It follows that $(V^b,N^b)\models \Exists{c}\bigwedge \pm\Psi$.
By semantics,
$((V^b)^{c\leftarrow u},N^b) \models\pm\Psi$ for some $u\in|V^b|$.
We get $((V^a)^{c\leftarrow v},N^a)\models \psi$ iff $((V^b)^{c\leftarrow u},N^b)\models \psi$
for all $\psi\in\Psi$.
By Lemma~\ref{lemma:bicond},
$((V)^{(c^1,c^2)\leftarrow (v,u)},N)\models \{ +(+\psi)^{\pi_1} \Leftrightarrow + (+\psi)^{\pi_2} \mid \psi\in\Psi\} $.
By satisfaction condition,
$((V)^{(z,c^1,c^2)\leftarrow (v,v,u)},N)\models \{ +(+\psi)^{\pi_1} \Leftrightarrow + (+\psi)^{\pi_2} \mid \psi\in\Psi\} $.
By satisfaction condition, since $(V,N)\models T_D$,
$((V)^{(z,c^1,c^2)\leftarrow (v,v,u)},N)\models T_D$. 
Since $(V^{z\leftarrow v},N)\models p$,
by satisfaction condition,
$((V)^{(z,c^1,c^2)\leftarrow (v,v,u)},N)\models p$. 
Since $v\in \pi_1^V$, $u\in \pi_2^V$ and the interpretations of $z$ and $c^1$ are $v$,
we obtain
$((V)^{(z,c^1,c^2)\leftarrow (v,v,u)},N)\models \{(\at{c^1}\pi_1),(\at{c^2}\pi_2),(z= c^1)\}$.
\end{itemize}
By compactness, $T_{D\cup\{c\}} \cup p \cup \{\at{z} \pi_1 \wedge z = c^1 \}$ is consistent, 
where $T_{D\cup\{c\}}$ is the set of all sentences from $T_C$ which contains only constants from $D^1\cup D^2 \cup \{c^1:\to \nom, c^2:\to\nom\}$.

Now let $E\subset C$ be any proper subset which includes $D$.
We define $E^i\coloneqq \chi_i^C(E)$ for each $i\in\{1,2\}$. 
Assuming that $T_E\cup p\cup \{\at{z}\pi_1\wedge z=c^1\}$ is consistent,
we prove that 
$T_{E\cup\{k\}}\cup p \cup \{\at{z}\pi_1\wedge z = c^1 \}$ is consistent for any $k\in C_\nom\setminus E_\nom$.
The proof is similar to the one above.

By compactness, 
$T_C\cup p \cup \{\at{z}\pi_1\wedge  z = c^1 \}$ is consistent.
Since $p\subseteq\Sen(\Delta^\diamond(C^\diamond,z))$ is an arbitrary set of cardinality strictly less than $\alpha$ consistent with $T_C$,
it follows that $T_C$ $\alpha$-omits $\Gamma_\nom$.
\end{proof}

\begin{notation}[Rigid types]
Let $z^1$ be a variable of sort $s^1\in S^\rigid_1$ and $z^2$ be a variable of sort $s^2\in S^\rigid_2$.
For each $i\in\{1,2\}$, we define a type in variable $z^i$:
\begin{center}
$\Gamma_{s^i}\coloneqq \{ z^i\neq c^i \mid c^i:\to s^i\in  C^i\}$.
\end{center}
\end{notation}

Notice that a Kripke structure $(V,N)$ over $\Delta^\diamond(C^\diamond)$ which omits $\Gamma_{s^i}$ has the carrier sets corresponding to the sort $s^i$ reachable by the constants of sort $s^i$ in $C^i$,
where $i\in\{1,2\}$.

\begin{proposition} \label{prop:omit-r}
$T_C$ $\alpha$-omits both types $\Gamma_{s^1}$ and $\Gamma_{s^2}$.
\end{proposition}

\begin{proof}
We show that $T_C$ omits $\Gamma_{s^1}$, since showing that $T_C$ omits $\Gamma_{s^2}$ is similar.
Moreover, we focus on the case when $s^1\in\chi_1(S^\rigid)$, since the case $s^1\not\in\chi_1(S^\rigid)$ is easy.

Let $p\subseteq\Sen(\Delta^\diamond(C^\diamond,z^1))$ be a set of sentences such that $\card(p)<\alpha$ and $T_C\cup p$ is consistent.
We define the subset of constants $D\subseteq C$ as follows:
(a)~for all rigid sorts $s\in \chi_1^{-1}(s^1)$, the set $D_s$ consists of all constants $c:\to s\in C$ such that either $c^1$ or $c^2$ occurs in $p$, and
(b)~for all rigid sorts $s\not \in \chi_1^{-1}(s^1)$, we have $D_s\coloneqq C_s$.
For each $i\in\{1,2\}$, we define $D^i\coloneqq \chi_i(D)\cup \{c: \to s\in C^i \mid s\in S_i^\rigid\setminus \chi_i(S^\rigid) \}$.
It follows that $p\subseteq \Sen(\Delta^\diamond(D^\diamond,z^1))$, 
where $(\Delta^\diamond(D^\diamond,z^1),\Phi_D^\diamond)$ is the relativized union of $\Delta^1(D^1)$ and $\Delta^2(D^2)$.

Let $T_D$ be the set of all sentences from $T_C$ which contains only constants from $D^1$ and $D^2$.
Since $T_C\cup p$ is consistent, its subset $T_D\cup p$ is consistent too.
Let $(V,N)$ be a Kripke structure over $\Delta^\diamond(D^\diamond)$ such that $(V,N)\models T_D$ and $(V,N^{z^1\leftarrow e})\models p$ for some possible world $v\in V_{\pi_1}$ and element $e\in N_{v,s^1}$.
According to our conventions,
$(V^1,N^1)= (V,N)\red_{\pi_1}$,
$(V^a,N^a)= (V,N)\red_{\chi_1^D}$,
$(V^2,N^2)= (V,N)\red_{\pi_2}$ and
$(V^b,N^b)= (V,N)\red_{\chi_2^D}$,
where $\chi_i^D$ denotes the restriction of $\chi_i^C$ to $\Delta(D)$ for each $i\in\{1,2\}$.
\begin{enumerate}[1)]
\item \label{prop:omit-r1} 
Let $s\in \chi_1^{-1}(s^1)$. 
Since $\card(D_s)<\alpha=\card(C_s)$, there exists $c \in C_s\setminus D_s$.
Let $\Psi\subseteq \Sen(\Delta(D,c))\setminus\Sen(\Delta(D))$ be a finite set of sentences.
We show that 
$T_D \cup  p \cup \{ z^1 = c^1 \} \cup \{ +(+\psi)^{\pi_1} \Leftrightarrow + (+\psi)^{\pi_2} \mid \psi\in\Psi\}$ is consistent, where $c^1:\to \chi_1(s)$ is the translation of $c:\to s$ along $\chi_1^C$.
The proof is similar to the first part of the proof of Proposition~\ref{prop:omit-n}.
By compactness, $T_{D\cup\{c\}} \cup p \cup \{ z^1 = c^1 \}$ is consistent, 
where $T_{D\cup\{c\}}$ is the set of all sentences from $T_C$ which contains only constants from $D^1\cup\{c^1:\to \chi_1(s) \}$ and $D^2\cup\{c^2:\to \chi_2(s)\}$.
\item \label{prop:omit-r2} 
Now let $E\subseteq C$ be an arbitrary subset of constants which includes $D$.
We define $E^i\coloneqq \chi_i^C(E)\cup \{c: \to s\in C^i \mid s\in S_i^\rigid\setminus \chi_i(S^\rigid) \}$ for each $i\in\{1,2\}$. 
Assuming that $T_E\cup p\cup \{ z^1=c^1\}$ is consistent,
we prove that 
$T_{E\cup\{d\}}\cup p \cup \{ z^1 = c^1 \}$ is consistent for any $d:\to s\in C\setminus E$.
The proof is similar to the one above.
\end{enumerate}
From (\ref{prop:omit-r1}) and (\ref{prop:omit-r2}), by compactness, 
$T_C\cup p \cup \{ z^1 = c^1 \}$ is consistent.
Since $p\subseteq\Sen(\Delta^\diamond(C^\diamond,z^1))$ is an arbitrary set of cardinality strictly less than $\alpha$ an consistent with $T_C$,
it follows that $T_C$ $\alpha$-omits $\Gamma_{s^1}$.
\end{proof}

All the preliminary results for proving Robinson consistency property are in place.
\begin{theorem}[Robinson consistency] \label{th:robinson}
Recall that $\chi_2$ is injective on sorts and nominals.
In addition, assume that $\chi_2$ protects flexible symbols.
Let $\Delta^1\stackrel{\upsilon_1}\to \Delta'\stackrel{\upsilon_2}\leftarrow\Delta^2$ be the pushout of $\Delta^1\stackrel{\chi_1}\leftarrow \Delta\stackrel{\chi_2}\to\Delta^2$.
Then $\upsilon_1(\Phi^1)\cup \upsilon^2(\Phi^2)$ is consistent.
\end{theorem}
\begin{figure}[h]\centering
\begin{tikzcd}
\Delta(C) \ar[rr,"\chi_1^C"] \ar[rd,"\chi_2^C "] & & \Delta^1(C^1) \ar[rd,dotted,"(\_)^{\pi_1}"] & \\ 
& \Delta^2(C^2)  \ar[rr,dotted,"(\_)^{\pi_2}", near start] & & \Delta^\diamond(C^\diamond) \\
& & & \\
\Delta \ar[uuu,hook] \ar[rr,"\chi_1" near start,dashed] \ar[rd,swap, "\chi_2"] & & \Delta^1 \ar[uuu,dashed,hook] \ar[dr,"\upsilon_1"]&\\ 
 & \Delta^2 \ar[uuu,hook] \ar[rr,"\upsilon_2"] & & \Delta'  \\
\end{tikzcd}
\caption{}
\end{figure}
\begin{proof}
By Proposition~\ref{prop:omit-n}, $T_C$ $\alpha$-omits $\Gamma_\nom$.
By Proposition~\ref{prop:omit-r},
$T_C$ $\alpha$-omits  $\Gamma_{s^1}$ and $\Gamma_{s^2}$ for all rigid sorts $s^1\in S^\rigid_1$ and $s^2\in S^\rigid_2$.
By ~\cite[Extended Omitting Types Theorem]{gai-hott}, 
there exists $(V,N)\in|\Mod(\Delta^\diamond(C^\diamond))|$ such that $(V,N)\models T_C$ and $(V,N)$ omits $\Gamma_\nom$, $\Gamma_{s^1}$ and $\Gamma_{s^2}$ for all rigid sorts $s^1\in S^\rigid_1$ and $s^2\in S^\rigid_2$.
Since $(V,N)\models \Phi^{\pi_1}\cup \Phi^{\pi_2}$, by satisfaction condition, 
$(V^1,N^1)\models \Phi^1$ and $(V^2,N^2)\models \Phi^2$.
Since $\Phi^1\models \chi_1(\Phi)$ and $\Phi^2\models \chi_2(\Phi)$, by satisfaction condition,
$(V^a,N^a)\models\Phi$ and $(V^b,N^b)\models\Phi$.
Since $\Phi$ is maximally consistent, 
$(V^a,N^a)\equiv(V^b,N^b)$.
Since $(V,N)$ omits $\Gamma_\nom$ and $\Gamma_{s^1}$ for all rigid sorts $s^1\in S^\rigid_1$,
$(V^1,N^1)$ is reachable by $C^1$.
Since $(V,N)$ omits $\Gamma_\nom$ and $\Gamma_{s^2}$ for all rigid sorts $s^2\in S^\rigid_2$, 
$(V^2,N^2)$ is reachable by $C^2$.
Since $(V^a,N^a)\equiv(V^b,N^b)$, $(V^a,N^a)$ is reachable by $C$ and $(V^2,N^2)$ is reachable by $C^2$,
by Lemma~\ref{lemma:lifting},
$(U^2,R^2)\equiv (V^2,N^2)$ for some $\chi_2^C$-expansion $(U^2,R^2)$ of $(V^a,N^a)$.
We define $(W^1,M^1)\coloneqq (V^1,N^1)\red_{\Delta^1}$ and $(W^2,M^2)\coloneqq (U^2,R^2)\red_{\Delta^2}$.
Since $(V^1,N^1)\models \Phi^1$ and $(U^2,R^2)\models \Phi^2$, 
by satisfaction condition,
$(W^1,M^1)\models \Phi^1$ and $(W^2,M^2)\models \Phi^2$.
Since $\Delta^1\stackrel{\upsilon_1}\to \Delta'\stackrel{\upsilon_2}\leftarrow\Delta^2$ is the pushout of $\Delta^1\stackrel{\chi_1}\leftarrow \Delta\stackrel{\chi_2}\to\Delta^2$ and $(W^1,M^1)\red_{\chi_1}=(W,M)=(W^2,M^2)\red_{\chi_2}$,
by \cite[Example 3.5]{dia-msc},
there exists a unique Kripke structure $(W',M')\in|\Mod(\Delta')|$ such that ${(W',M')\red_{\upsilon_1}}=(W^1,M^1)$ and $(W',M')\red_{\upsilon_2}=(W^2,M^2)$.
By satisfaction condition, $(W',M')\models \upsilon_1(\Phi^1)\cup \upsilon^2(\Phi^2)$.
\end{proof}

In many-sorted first-order logic, if one of the signature morphisms in the span is injective on sorts then the corresponding pushout is a CI square~\cite{DBLP:journals/sLogica/GainaP07}.
Lemma~\ref{lemma:counter-1} shows that in $\HFOL$, this condition is also necessary.
The following two examples focus on the second condition, the protection of flexible symbols.

\begin{example}\label{ex:counter-2}
Let
$\Delta^1\stackrel{\chi_1}\hookleftarrow \Delta\stackrel{\chi_2}\hookrightarrow\Delta^2$ 
be a span of inclusions such that
\begin{itemize}
\item $\Delta$ has one nominal $\{k\}$, one flexible sort $\{s\}$ and one constant $\{c:\to s\}$;
\item $\Delta^1$ has nominals $\{k,k_1\}$,
one rigid sort $\{s\}$ and one flexible constant $\{c:\to s\}$; 
\item $\Delta^2$ has two nominals $\{k,k_2\}$,
one flexible sort $s$, and two flexible constants $\{c:\to s, c_2:\to s\}$.
\end{itemize}
Let $\Delta^1\stackrel{\upsilon_1}\hookrightarrow \Delta' \stackrel{\upsilon_2}\hookleftarrow\Delta^2$ be a pushout of the above span such that
\begin{itemize}
\item $\Delta'$ consists of three nominals $\{k,k_1,k_2\}$,
one rigid sort $\{s\}$, and 
two flexible constants $\{c:\to s, c_2:\to s\}$.
\end{itemize}
\end{example}

In Example~\ref{ex:counter-2}, the signature morphism $\chi_2$ adds a new constant $c_2:\to s$ on the flexible sort $s$, which means that flexible symbols are not protected. 
The following lemma shows that the pushout constructed above is not a CI square.

\begin{lemma}\label{lemma:counter-2}
The pushout described in Example~\ref{ex:counter-2} is not a CI square.
\end{lemma}

\begin{proof}
Let $\Phi_1\coloneqq\{\Forall{y}c=y\in\Sen(\Delta^1)\}$ and 
$\Phi_2\coloneqq\{c_2=c\in\Sen(\Delta^2)\}$.
It is straightforward to show $\Phi_1\models \Phi_2$.
Suppose towards a contradiction that there exists $\Phi\subseteq\Sen(\Delta)$ such that $\Phi_1\models\Phi$ and $\Phi\models\Phi_2$.

Let $(W^1,M^1)$ be the Kripke structure over $\Delta^1$ that consists of 
one possible  world $w$, which means that $W^1_k=W^1_{k_1}=w$ and
$M^1_w$ is the single-sorted algebra consisting of one element $M^1_{w,s}=\{e\}$, 
which means that $M^1_{w,c}=e$.
Obviously, $(W^1,M^1)\models \Forall{x}x=c$.
By the satisfaction condition, 
$(W^1,M^1)\red_\Delta\models \Phi$.
Let $(W,M)$ be the Kripke structure over $\Delta$ obtained from $(W^1,M^1)\red_\Delta$ by adding a new flexible element $d$ of sort $s$.
Since $d$ is an unreachable element, by Lemma~\ref{lemma:reach-equiv}, $(W,M)\equiv(W^1,M^1)\red_\Delta$.
It follows that $(W,M)\models\Phi$.
Let $(W^2,M^2)$ be the expansion of $(W,M)$ to $\Delta^2$ which interprets $c_2:\to s$ as $d$.
By the satisfaction condition, $(W^2,M^2)\models\Phi$.
Since $\Phi\models \Phi_2$, we have $(W^2,M^2)\models \Phi_2$, 
which is a contradiction, as $M^2_{w,c}=e\neq d = M^2_{w,c_2}$.
\end{proof}

We give another example of pushout which is not a CI square.
\begin{example}\label{ex:counter-3}
Let
$\Delta^1\stackrel{\chi_1}\hookleftarrow \Delta\stackrel{\chi_2}\hookrightarrow\Delta^2$ 
be a span of inclusions such that
\begin{itemize}
\item~$\Delta$ has one nominal $\{k\}$, 
one flexible sort $\{Nat\}$ and  
two flexible function symbols $\{0:\to Nat,succ:Nat \to Nat\}$;
\item~$\Delta^1$ has two nominals $\{k,k_1\}$,
one rigid sort $\{Nat\}$ and
three flexible function symbols $\{0: \to Nat, succ:Nat \to Nat, \texttt{\_+\_}: Nat~Nat\to Nat \}$;
\item~$\Delta^2$ has two nominals $\{k,k_2\}$,
two rigid sorts $\{Nat,List\}$, 
four flexible operations
$\{0:\to Nat, succ:Nat \to Nat, nil:\to List, \texttt{\_|\_}: Nat~List\to List\}$.
\end{itemize}
Let $\Delta^1\stackrel{\upsilon_1}\hookrightarrow \Delta' \stackrel{\upsilon_2}\hookleftarrow\Delta^2$ be a pushout of the above span such that
\begin{itemize}
\item $\Delta'$ has three nominals $\{k,k_1,k_2\}$,
two rigid sorts $\{Nat,List\}$ and 
five flexible function symbols 
$\{0:\to Nat, succ:Nat\to Nat, \texttt{\_+\_}: Nat~Nat\to Nat, nil:\to List, cons:Nat~List\to List\}$.
\end{itemize}
\end{example}

In Example~\ref{ex:counter-3}, the signature morphism $\chi_2$ does not preserve the flexible sort $Nat$.
\begin{lemma}\label{lemma:counter-3}
The pushout described in Example~\ref{ex:counter-3} is not a CI square.
\end{lemma}

\begin{proof}
Let $\Phi^1$ be the set of $\Delta^1$-sentences which consists of 
\begin{itemize}
\item $\Forall{{x:Nat}}succ(succ(x))=x$, 
\item $\Forall{{x:Nat}}0 + x = x$, and
\item $\Forall{{x:Nat},{y:Nat}} succ(y) + x = succ(y + x)$.
\end{itemize}
Let $\Phi^2\coloneqq \{\Forall{x:Nat}succ(succ(x))=x \}$. 
Suppose towards a contradiction that there exists a set of $\Delta$-sentences $\Phi$ such that 
$\Phi^1\models\Phi$ and $\Phi\models\Phi^2$.

Let $(W^1,M^1)$ be a Kripke structure over $\Delta^1$ that consists of 
two possible worlds $\{w_1,w_2\}$,
both $M^1_{w_1}$ and $M^1_{w_2}$ are the quotient algebra $\mathbb{Z}_2$.
Obviously, $(W^1,M^1)\models \Phi^1$.
Since $\Phi^1\models \Phi$, $(W^1,M^1)\models \Phi$.
By the satisfaction condition, $(W^1,M^1)\red_\Delta\models\Phi$.
Let $(W,M)$ obtained by adding a new element $\widehat{2}$ of sort $Nat$ such that $M_{w_i,succ}(\widehat{2})=\widehat{0}$ for each $i\in\{1,2\}$.
Since $\widehat{2}$ is an unreachable element, by Lemma~\ref{lemma:reach-equiv}, $(W,M)\equiv(W^1,M^1)\red_\Delta$.
Let $(W^2,M^2)$ be the expansion of $(W,M)$ to $\Delta^2$ which interprets 
$List$ in both worlds as the set of all lists with elements from $\{\widehat{0},\widehat{1},\widehat{2}\}$.
By the satisfaction condition, $(W^2,M^2)\models \Phi$.
Since $M_{w_1,succ} (M_{w_1,succ}(\widehat{2})) = \widehat{1}$, 
we have $(W^2,M^2)\not \models \Phi^2$, 
which contradicts $\Phi\models\Phi^2$.
\end{proof}
\section{Conclusions}
Lemma~\ref{lemma:counter-2} and Lemma~\ref{lemma:counter-3} show that not only injectivity on sorts but also  protection of flexible symbols is necessary for interpolation in $\HFOL$.

Recall that $\RFOHL$ signatures form a subcategory of $\HFOL$ signatures. 
It is not difficult to check that the subcategory of $\RFOHL$ signatures is closed under pushouts.
It follows that Theorem~\ref{th:robinson} is applicable to $\RFOHL$.
Since intersection-union square of signature morphism are, in particular, pushouts, 
our results cover the ones obtained in \cite{ArecesBM03}.

Similarly, $\HPL$ signatures form a subcategory of $\HFOL$ signatures. 
It is not difficult to check that the subcategory of $\HPL$ signatures is closed under pushouts.
It follows that Theorem~\ref{th:robinson} is applicable to $\HPL$.
Since Theorem~\ref{th:robinson} is derived from Omitting Types Theorem, 
our results rely on quantification over possible worlds.
Therefore, the present work does not cover the interpolation result from \cite{ArecesBM01}, 
which is applicable to hybrid propositional logic without quantification.

$\HFOL$ and $\HFOLS$ have the same signatures and Kripke structures.
By \cite[Lemma 2.20]{gai-godel}, $\HFOLS$ and $\HFOL$ have the same expressivity power.
The relationship between $\HFOL$ and $\HFOLS$ is similar to the relationship between first-order logic and unnested first-order logic, which allows only terms of depth one~\cite{DBLP:books/daglib/0080659}.
Therefore, a square of $\HFOL$ signature morphisms is a CI square in $\HFOL$ iff 
it is a CI square in $\HFOLS$.

Proposition~\ref{def:rob} does not give an effective construction of an interpolant, which is an open problem.

\bibliographystyle{aiml22}
\bibliography{aiml22}

\begin{thebibliography}{10}
\expandafter\ifx\csname url\endcsname\relax
  \def\url#1{\texttt{#1}}\fi
\expandafter\ifx\csname urlprefix\endcsname\relax\def\urlprefix{URL }\fi
\newcommand{\enquote}[1]{``#1''}

\bibitem{ArecesB01}
Areces, C. and P.~Blackburn, \emph{{Bringing them all Together}}, Journal of
  Logic and Computation \textbf{11} (2001), pp.~657--669.

\bibitem{ArecesBM01}
Areces, C., P.~Blackburn and M.~Marx, \emph{Hybrid logics: characterization,
  interpolation and complexity}, Journal of Symbolic Logic \textbf{66} (2001),
  pp.~977--1010.

\bibitem{ArecesBM03}
Areces, C., P.~Blackburn and M.~Marx, \emph{Repairing the interpolation theorem
  in quantified modal logic}, Ann. Pure Appl. Log. \textbf{124} (2003),
  pp.~287--299.

\bibitem{DBLP:conf/wollic/BlackburnMMH19}
Blackburn, P., M.~A. Martins, M.~Manzano and A.~Huertas, \emph{Rigid
  first-order hybrid logic}, in: R.~Iemhoff, M.~Moortgat and R.~J. G.~B.
  de~Queiroz, editors, \emph{Logic, Language, Information, and Computation -
  26th International Workshop, WoLLIC 2019, Utrecht, The Netherlands, July 2-5,
  2019, Proceedings},  Lecture Notes in Computer Science  \textbf{11541}
  (2019), pp. 53--69.

\bibitem{DBLP:journals/tcs/Borzyszkowski02}
Borzyszkowski, T., \emph{Logical systems for structured specifications}, Theor.
  Comput. Sci. \textbf{286} (2002), pp.~197--245.

\bibitem{cod-h}
Codescu, M., \emph{Hybridisation of institutions in {HETS} (tool paper)}, in:
  M.~Roggenbach and A.~Sokolova, editors, \emph{8th Conference on Algebra and
  Coalgebra in Computer Science, {CALCO} 2019, June 3-6, 2019, London, United
  Kingdom},  LIPIcs  \textbf{139} (2019), pp. 17:1--17:10.

\bibitem{tutu-iilp}
\c{T}u{t}u, I. and J.~L. Fiadeiro, \emph{From conventional to
  institution-independent logic programming}, J. Log. Comput. \textbf{27}
  (2017), pp.~1679--1716.

\bibitem{dia-book}
Diaconescu, R., \enquote{{Institution-independent Model Theory},} Studies in
  Universal Logic, Birkh{\" a}user, Basel, 2008, 1 edition.

\bibitem{dia-qvh}
Diaconescu, R., \emph{{Quasi-varieties and initial semantics for hybridized
  institutions}}, Journal of Logic and Computation \textbf{26} (2016),
  pp.~855--891.

\bibitem{dia-msc}
Diaconescu, R. and A.~Madeira, \emph{{Encoding Hybridised Institutions into
  First-Order Logic}}, Mathematical Structures in Computer Science \textbf{26}
  (2016), pp.~745--788.

\bibitem{DBLP:books/daglib/0080659}
Ebbinghaus, H., J.~Flum and W.~Thomas, \enquote{Mathematical logic {(2.} ed.),}
  Undergraduate Texts in Mathematics, Springer, 1994.

\bibitem{gog-ins}
Goguen, J. and R.~Burstall, \emph{Institutions: Abstract model theory for
  specification and programming}, Journal of the Association for Computing
  Machinery \textbf{39} (1992), pp.~95--146.

\bibitem{gai-int}
G\u{a}in\u{a}, D., \emph{{Interpolation in logics with constructors}},
  Theoretical Computer Science \textbf{474} (2013), pp.~46--59.

\bibitem{gai-bir}
G\u{a}in\u{a}, D., \emph{Birkhoff style calculi for hybrid logics}, Formal Asp.
  Comput. \textbf{29} (2017), pp.~805--832.

\bibitem{gai-dls}
G\u{a}in\u{a}, D., \emph{{Downward L\"owenheim-Skolem Theorem and interpolation
  in logics with constructors}}, Journal of Logic and Computation \textbf{27}
  (2017), pp.~1717--1752.

\bibitem{gai-her}
G\u{a}in\u{a}, D., \emph{Foundations of logic programming in hybrid logics with
  user-defined sharing}, Theor. Comput. Sci. \textbf{686} (2017), pp.~1--24.

\bibitem{gai-godel}
G\u{a}in\u{a}, D., \emph{Forcing and calculi for hybrid logics}, Journal of the
  Association for Computing Machinery \textbf{67} (2020), pp.~25:1--25:55.

\bibitem{gai-hott}
G\u{a}in\u{a}, D., G.~Badia and T.~Kowalski, \emph{Omitting types theorem in
  hybrid-dynamic logics with rigid symbols}, math \textbf{abs/2203.08720}
  (2022).

\bibitem{gai-cbl}
G\u{a}in\u{a}, D., K.~Futatsugi and K.~Ogata, \emph{{Constructor-based
  Logics}}, J. {UCS} \textbf{18} (2012), pp.~2204--2233.

\bibitem{gai-com}
G\u{a}in\u{a}, D. and M.~Petria, \emph{Completeness by forcing}, Journal of
  Logic and Computation \textbf{20} (2010), pp.~1165--1186.

\bibitem{DBLP:journals/sLogica/GainaP07}
G\u{a}in\u{a}, D. and A.~Popescu, \emph{An institution-independent proof of the
  robinson consistency theorem}, Stud Logica \textbf{85} (2007), pp.~41--73.

\bibitem{lind78}
Lindstr\"om, P., \emph{Omitting uncountable types and extensions of elementary
  logic}, Theoria - a Swedish Journal of Philosophy \textbf{44} (1978),
  pp.~152---156.

\bibitem{martins}
Martins, M.~A., A.~Madeira, R.~Diaconescu and L.~S. Barbosa,
  \emph{Hybridization of institutions}, in: A.~Corradini, B.~Klin and
  C.~C{\^{\i}}rstea, editors, \emph{Algebra and Coalgebra in Computer Science -
  4th International Conference, {CALCO} 2011, Proceedings},  Lecture Notes in
  Computer Science  \textbf{6859} (2011), pp. 283--297.

\bibitem{Petria07}
Petria, M., \emph{{An Institutional Version of G{\"{o}}del's Completeness
  Theorem}}, in: T.~Mossakowski, U.~Montanari and M.~Haveraaen, editors,
  \emph{Algebra and Coalgebra in Computer Science, Second International
  Conference, {CALCO} 2007, Bergen, Norway, August 20-24, 2007, Proceedings},
  Lecture Notes in Computer Science  \textbf{4624} (2007), pp. 409--424.

\bibitem{rob}
Robinson, A., \emph{A result on consistency and its application to thetheory of
  definition}, Indagationes Mathematicae (Proceedings) \textbf{59} (1956),
  pp.~47--58.
\newline\urlprefix\url{https://www.sciencedirect.com/science/article/pii/S138572585650008X}

\bibitem{tar-bit}
Tarlecki, A., \emph{{Bits and Pieces of the Theory of Institutions}},  ,
  \textbf{240}, Springer, 1986 pp. 334--360.

\end{thebibliography}

\end{document}